\documentclass{article}                  
\usepackage{graphicx}
\usepackage{amsmath,amssymb,amsthm,bm}
\usepackage{hyperref}
\usepackage{xspace}
\usepackage{xcolor}
\usepackage{dsfont}
\usepackage[ruled,vlined,linesnumbered]{algorithm2e}
\usepackage{mathtools}
\usepackage{booktabs}
\usepackage{subcaption}
\usepackage{authblk}
\SetKwComment{Comment}{// }{}
\SetKwFunction{AlphaBeta}{AlphaBeta}
\SetKwFunction{Assign}{assign}
\SetKwFunction{Unassign}{unassign}
\SetKw{Continue}{continue}
\SetKw{Break}{break}
\SetKw{KwDownTo}{down to}
\SetKwInput{KwData}{Input}
\SetKwInput{KwResult}{Output}

\newtheorem{theorem}{Theorem}
\newtheorem{definition}{Definition}
\newtheorem{corollary}{Corollary}

\newcommand{\score}{\textnormal{\texttt{score}}}
\newcommand{\valuevar}{\textnormal{\texttt{value}}}
\newcommand{\finished}[1][]{%
  \texttt{levelFinished}%
  \ifx\relax#1\relax
  \else
  [#1]%
  \fi
}

\newcommand{\polarity}{\textnormal{\texttt{$\pi$}}}
\newcommand{\targetlevel}{\textnormal{\texttt{targetLevel}}}
\newcommand{\implied}{\textnormal{\texttt{impliedVariable}}}

\newcommand{\true}{\textnormal{\texttt{true}}}
\newcommand{\false}{\textnormal{\texttt{false}}}
\newcommand{\none}{\textnormal{\texttt{none}}}

\RequirePackage{etex} 

\usepackage{graphicx,psfrag,epsfig,color,colortbl,rotate}

\usepackage[ruled,vlined]{algorithm2e}
\usepackage{listings}
\usepackage{linegoal}

\usepackage[skins]{tcolorbox}

\tcbset{commonstyle/.style={boxrule=0pt,sharp corners,enhanced jigsaw,nobeforeafter,boxsep=0pt,left=\fboxsep,right=\fboxsep}}

\newtcolorbox{mycolorbox}[1][]{commonstyle,#1}

\usepackage{float}

\makeatletter
\newcommand{\nosemic}{\renewcommand{\@endalgocfline}{\relax}}
\newcommand{\dosemic}{\renewcommand{\@endalgocfline}{\algocf@endline}}
\let\oldnl\nl
\newcommand{\nonl}{\renewcommand{\nl}{\let\nl\oldnl}}
\makeatother

\newtheorem{assumption}{Assumption}

\newcommand{\Alphabeta}{alpha-beta\xspace}

\newcommand{\Mibs}{\texttt{MibS}\xspace}
\newcommand{\Yasol}{\texttt{Yasol}\xspace}
\newcommand{\Equips}{\texttt{EQuIPS}\xspace}
\newcommand{\Cplex}{\texttt{CPLEX}\xspace}
\newcommand{\Domain}{\ensuremath{\mathcal{D}}}

\newcommand{\Domainxi}[1][]{%
  \Domain_{\xivec}%
  \ifx\relax#1\relax
  \else
    ^{#1}%
  \fi
}

\newcommand{\Domainx}[1][]{%
  \Domain_{\xvec}%
  \ifx\relax#1\relax
  \else
    ^{#1}%
  \fi
}

\newcommand{\Legal}{\ensuremath{\mathcal{A}}}

\newcommand{\Legalx}[1][]{%
  \Legal_{\xvec}%
  \ifx\relax#1\relax
  \else
    ^{#1}%
  \fi
}

\newcommand{\Legalxi}[1][]{%
  \Legal_{\xivec}%
  \ifx\relax#1\relax
  \else
    ^{#1}%
  \fi
}

\newcommand{\xvec}[1][]{%
  \pmb{x}%
  \ifx\relax#1\relax
  \else
    ^{#1}%
  \fi
}

\newcommand{\xvect}[1][]{%
  \tilde{\xvec}%
  \ifx\relax#1\relax
  \else
    ^{#1}%
  \fi
}

\newcommand{\xvech}[1][]{%
  \hat{\xvec}%
  \ifx\relax#1\relax
  \else
    ^{#1}%
  \fi
}

\newcommand{\xivec}[1][]{%
  \pmb{\xi}%
  \ifx\relax#1\relax
  \else
    ^{#1}%
  \fi
}

\newcommand{\xivect}[1][]{%
  \tilde{\xivec}%
  \ifx\relax#1\relax
  \else
    ^{#1}%
  \fi
}

\newcommand{\xivech}[1][]{%
  \hat{\xivec}%
  \ifx\relax#1\relax
  \else
    ^{#1}%
  \fi
}

\newcommand{\Ax}{\ensuremath{A^{\xvec}}}
\newcommand{\Axi}{\ensuremath{A^{\xivec}}}
\newcommand{\Gx}{\ensuremath{G^{\xvec}}}
\newcommand{\Gxi}{\ensuremath{G^{\xivec}}}
\newcommand{\bvec}{\ensuremath{\pmb{b}}}
\newcommand{\hvec}{\ensuremath{\pmb{h}}}
\newcommand{\cvec}{\ensuremath{\pmb{c}}}
\newcommand{\dvec}{\ensuremath{\pmb{d}}}
\newcommand{\lvec}{\ensuremath{\pmb{l}}}
\newcommand{\uvec}{\ensuremath{\pmb{u}}}

\newcommand{\MCNBasic}{\textbf{\Yasol}$_{B}$\xspace}
\newcommand{\MCNDD}{\textbf{\Yasol}$_{{DD}}$\xspace}
\newcommand{\MCNBaggio}{\textbf{\texttt{MCN}}\xspace}
\newcommand{\MCNEQUIPS}{\textbf{\Equips}\xspace}

\begin{document}

\title{A Search-Based Algorithm for Multistage Robust Discrete Linear Optimization and its Implementation
\thanks{Partially funded by Deutsche Forschungsgemeinschaft (DFG, German
35 Research Foundation) - 534441421}
}

\author[1]{Michael Hartisch}
\author[2]{Ulf Lorenz}

\affil[1]{University of Passau, Passau, Germany}
\affil[2]{University of Siegen, Siegen, Germany}

\maketitle

\begin{abstract}

Robust optimization models that allow decisions to adapt to newly revealed information remain computationally challenging and are often limited to two decision stages or are addressed heuristically or via approximation schemes. This challenge is further amplified when uncertainty sets depend on earlier decisions. Conventional approaches to such problems often require problem-specific reformulations and tailored algorithms, highlighting the need for an exact model-and-run approach. We address this gap by presenting an algorithmic framework for solving multistage robust discrete linear optimization problems. 
This exact approach adapts and integrates ideas from game-tree search, quantified conflict-driven clause learning, and integer linear programming to explore the minimax tree and is implemented in the
open-source solver Yasol. The solver supports discrete decision variables across multiple stages, discrete uncertainty sets with polyhedral or decision-dependent structure, and mixed-integer recourse in the final stage. It can be used as a standalone executable or through a C++ library interface, allowing it to serve as a reference solver for exact methods in multistage robust discrete linear optimization. In an extensive computational study, we demonstrate the capabilities of our solver. The results show competitive performance against a reference approach from the literature for a protection-interdiction problem, as well as strong performance on bilevel benchmark instances. For a multistage cutting-stock problem, we demonstrate the substantial loss in solution quality that can arise from neglecting the multistage structure and show that our solver significantly outperforms decision-rule approximations in terms of computational efficiency while certifying the true multistage optimum.

\end{abstract}

\section{Introduction}
\subsection{Motivation}
Many decision problems in optimization, verification, and planning involve sequential decisions under uncertainty or adversarial behavior. From a theoretical perspective, such problems are often computationally challenging and can occupy higher levels of the complexity hierarchy. Nevertheless, decades of research in integer linear programming, satisfiability solving, and game-tree search have shown that carefully engineered algorithms can successfully solve large and practically relevant instances despite these worst-case hardness results.

This paper presents an algorithmic framework for solving multistage robust discrete linear optimization problems and its open-source software implementation Yasol. The development of Yasol started more than a decade ago, evolving from a solver for Quantified Boolean Formula (QBF), to a solver for Quantified Integer Programs (QIPs), and further moving the boundaries to account for multistage robust integer linear optimization problems with decision-dependent uncertainty. 
This evolution allowed the solver to thrive from ideas from game-tree search, quantified constraint solving, and mixed-integer linear programming (MIP). The solver supports multistage problems with discrete decision variables, mixed-integer recourse in the final decision stage, and discrete uncertainty sets that may be scenario-based, polyhedral, or decision-dependent. Rather than relying on reformulation, decomposition, or approximation techniques, our approach directly explores the implicit minimax tree induced by the problem structure to compute exact solutions.

\subsection{Multistage Robust Integer Linear Optimization\label{sec::problem}}

 Let $\xvec \in \mathbb{Z}^n$, $n\in\mathbb{N}$, denote a vector of integer \textit{decision variables}, partitioned into $T+1$ stages, $T\in\mathbb{N}$, i.e., 
$
\xvec = \bigl(\xvec[0],\xvec[1],\ldots,\xvec[T]\bigr),
$
where $\xvec[t] \in \mathbb{Z}^{n_t}$ represents the decision made by the decision-maker at stage $t\in\{0,1,\ldots,T\}$, with $\sum_{t=0}^T n_t=n$. Each component of $\xvec$ is bounded by a corresponding lower and upper bound. Accordingly, we define the decision domain as $\Domainx = \{\xvec\in\mathbb{Z}^n \mid \lvec_{\xvec}\leq\xvec\leq \uvec_{\xvec}\}$,
where $\lvec_{\xvec},\uvec_{\xvec}\in\mathbb{Z}^n$ are the vectors of lower and upper bounds.
Similarly, let $\xivec \in \mathbb{Z}^k$, $k\in \mathbb{N}$, denote the vector of integer \textit{adversary variables}, which is partitioned into $T$ stages, i.e., $\xivec=\bigl(\xivec[1],\ldots,\xivec[T]\bigr)$, where $\xivec[t]\in \mathbb{Z}^{k_t}$ is the vector representing the adversary decision in stage $t\in \{1,\ldots,T\}$, with  $\sum_{t=1}^T k_t=k$. Adversary variables are bounded within $\Domainxi =\{\xivec\in\mathbb{Z}^k \mid \lvec_{\xivec}\leq\xivec\leq \uvec_{\xivec}\}$, with lower and upper bounds $\lvec_{\xivec},\uvec_{\xivec}\in\mathbb{Z}^n$. We write $\Domainx[t]$ and $\Domainxi[t]$ when referring to the bounded domain of decision and adversary variables of stage $t$, respectively. We note that throughout this paper, we use the term ``variable'' for both the decision-maker's decisions and the adversary's choices. Thus, quantities that would typically be viewed as uncertain parameters in robust optimization are treated as variables controlled by the adversary.

In addition to the bounds, the decision-maker is required to satisfy a system of linear constraints
\[
\Ax\xvec + \Axi\xivec \leq \bvec,
\]
with $\Ax \in \mathbb{R}^{m\times n}$, $\Axi \in \mathbb{R}^{m\times k}$, $\bvec \in \mathbb{R}^m$, $m\in\mathbb{N}$,
which we refer to as the \textit{decision-maker constraint system}. 
The decision-maker wants to maximize the linear objective function
$\cvec^\top \xvec + \dvec^\top \xivec$, 
where $\cvec \in \mathbb{R}^n$, $\dvec \in \mathbb{R}^k$, while the adversary seeks to minimize it. The resulting multistage robust integer linear problem with discrete box uncertainty can be compactly written as
\begin{equation}
\label{MRILP-Box}
\max_{\xvec[0]\in \Domainx[0]} \min_{\xivec[1]\in \Domainxi[1]}
\max_{\xvec[1]\in \Domainx[1]} \cdots
\min_{\xivec[T]\in \Domainxi[T]} \max_{\xvec[T]\in \Domainx[T]}
\; \cvec^\top \xvec + \dvec^\top \xivec
\quad
\text{s.t. } \Ax\xvec + \Axi\xivec \leq \bvec . \tag{MRILP-Box}
\end{equation}
This optimization problem can be viewed as a zero-sum game between two agents---the decision-maker and the adversary---and aims at finding the best strategy for assigning the decision variables such that the decision-maker constraint system is satisfied for every assignment of the adversary variables. Here, a strategy specifies how decision variables adapt to revealed uncertainty. The value of a strategy is determined by its worst-case outcome over all assignments of the adversary variables.

To extend the modeling capabilities, a second constraint system is used to restrict the adversary's admissible choices, and thereby defining a more sophisticated uncertainty set. We refer to this system as the \textit{adversary constraint system} and denote it by \[
\Gx\xvec + \Gxi\xivec \leq \hvec,
\]
with $\Gx \in \mathbb{R}^{q\times n}$, $\Gxi \in \mathbb{R}^{q\times k}$, $\hvec \in \mathbb{R}^q$, $q\in\mathbb{N}$. This leads to two relevant uncertainty structures: If $\Gx=0$, the adversary constraint system defines a polyhedral uncertainty set. If $\Gx\neq 0$, the admissible adversary assignments depend on decisions of the decision-maker, giving rise to a decision-dependent uncertainty set.

To obtain a well-defined problem, we assume that each agent must first respect her own constraint system. In particular, an agent may not violate her own constraints merely to force a violation of the other agent's constraint system. To formalize this requirement, we define an admissible variable assignment as follows.

\begin{definition}[Admissible Assignment\label{def::legal}]
For the decision-maker's assignment at stage $t\in\{0,1,\ldots,T\}$, given the prefix assignment $\xvect[0],\xivect[1],\xvect[1],\ldots,\xivect[t]$ describing the progression up to this stage, the set of \emph{admissible assignments} for the decision-maker is given by
\[ 
\Legalx[t](\xvect[0],\xivect[1],\xvect[1],\ldots,\xivect[t])
=
\left\{
\xvech[t]\in\Domainx[t]
\;\middle|\;
\begin{aligned}
&\exists (\xvec,\xivec) \in  \Domainx\times\Domainxi: \\
&\xvec[t]
=
\xvech[t], 
\\
&\xvec[p]
=
\xvect[p] \ \forall p \in \{0,1\ldots,t-1\},
\\
&\xivec[p]
=
\xivect[p] \ \forall p \in \{1\ldots,t\},
\\
&\Ax\xvec+\Axi\xivec\leq\bvec
\end{aligned}
\right\}.
\]
Similarly, for the adversary's assignment at stage $t\in\{1,\ldots,T\}$, given the prefix assignment $\xvect[0],\xivect[1],\xvect[1],\ldots,\xvect[t-1]$, the set of admissible assignments for the adversary is given by
\[ 
\Legalxi[t](\xvect[0],\xivect[1],\xvect[1],\ldots,\xvect[t-1])
=
\left\{
\xivech[t]\in\Domainxi[t]
\;\middle|\;
\begin{aligned}
&\exists (\xvec,\xivec) \in  \Domainx\times\Domainxi: \\
&\xivec[t]
=
\xivech[t], 
\\
&\xvec[p]
=
\xvect[p] \ \forall p \in \{0,1\ldots,t-1\},
\\
&\xivec[p]
=
\xivect[p] \ \forall p \in \{1\ldots,t-1\},
\\
&\Gx\xvec+\Gxi\xivec\leq\hvec
\end{aligned}
\right\}.
\]
We write $\Legalx[t]$ and $\Legalxi[t]$ if the dependence on the assignment of previous stages is clear. Similarly, with slight abuse of notation, we write $\Legalx[i]$ and $\Legalxi[j]$ when referring to single variables $i\in\{1,\ldots,n\}$ and $j\in\{1,\ldots,k\}$.
\end{definition}
Admissibility ensures that after assigning the variables of stage $t$, there exists a continuation for the upcoming stages within the bounded domains of both agents such that the constraint system of the acting agent can be fulfilled.

\begin{assumption}\label{ass::1} $
\{(\xvec,\xivec) \in  \Domainx\times\Domainxi \mid \Gx\xvec+\Gxi\xivec\leq\hvec\} \neq \emptyset$.
\end{assumption}

\begin{assumption} \label{ass::2}
All entries of $\Gx$ corresponding to $\xvec[T]$  are zero.
\end{assumption}
Assumption~\ref{ass::1} demands that there exists a progression satisfying the adversary constraint system, which essentially means that the uncertainty set must not be empty. 
Assumption~\ref{ass::2} requires the uncertainty set to be independent of the final-stage decision. Hence, while the final decision may adapt to the realized adversary action, it cannot be used to violate the adversary constraint system. This is also fulfilled in the case if there is no final stage of the decision-maker. This assumption also guarantees that, whenever the agents make admissible assignments at each stage, the resulting full variable vector satisfies both constraint systems.  Using this, we can state the Multistage Robust Integer Linear Program with decision-dependent uncertainty as follows
\begin{equation}
\label{problem}
\max_{\xvec[0]\in \Legalx[1]} \min_{\xivec[1]\in \Legalxi[1]} \cdots
\min_{\xivec[T]\in \Legalxi[T]} \max_{\xvec[T]\in \Legalx[T]}
\;\; \cvec^\top \xvec + \dvec^\top \xivec\ . \tag{MRILP} 
\end{equation}
If all stages are completed with admissible assignments, the value of the complete assignment is $\cvec^\top \xvec + \dvec^\top \xivec$. If, at some stage, no admissible assignment exists, the value is defined to be $-\infty$ for stages of the decision-maker and $+\infty$ for adversary stages. 

\subsection{Computational challenges}
From a theoretic point of view, the endeavor of tackling general problems of type \eqref{problem} seems hopeless.
Already QBF with alternation depth $T$ is complete for the $T$-th level of the polynomial hierarchy~\cite{wrathall1976complete}, which also extends to QIP~\cite{chistikov2017complexity}.  Problem~\eqref{problem} inherits this alternating quantifier structure and additionally incorporates optimization over integer variables and linear constraint systems. Consequently, these problems combine the difficulty of quantified programming with that of integer optimization. This suggests that, from a theoretical standpoint, one should not expect generally efficient solution methods. 
However, complexity theory provides insight primarily into worst-case behavior. Moreover, technological progress and sustained research have led to remarkable tools for NP-hard problems that appeared computationally intractable only a few decades ago. Hence, theoretical difficulty should not deter us from investigating algorithmic approaches. To solve Problem~\eqref{problem}, we can draw on algorithmic techniques from related areas, including game-tree search, robust optimization, QBF solving, and mixed-integer programming. These methods are adapted cautiously to the specific structure of Problem~\eqref{problem} and complemented by problem-specific techniques designed to improve practical performance despite the underlying worst-case complexity.

\subsection{Related Work\label{sec:relWork}}

\paragraph{Robust and Multistage Optimization.}
Robust optimization studies decision making under uncertainty by optimizing against worst-case realizations of uncertain parameters within prescribed uncertainty sets \cite{ben2002robust,gabrel2014recent,goerigk2024introduction}. While static robust optimization is generally well understood, multistage and adjustable robust opt\-imization---where decisions may adapt to uncertainty revealed over time---remain computationally challenging. Consequently, most existing approaches focus on two-stage models or rely on approximation schemes, such as decision rules, to achieve tractability \cite{ben2004adjustable,bertsimas2018binary,yanikouglu2019survey}. 
Several algorithmic paradigms for multistage robust optimization have been explored, including dynamic programming \cite{shapiro2011dynamic}, decomposition techniques  \cite{thiele2009robust}, iterative refinement of the uncertainty set \cite{postek2016multistage}, or partition-based approaches \cite{bertsimas2016multistage}. While effective for specific problem classes, these methods often suffer from scalability issues, rely on restrictive assumptions, or require problem-specific reformulations. As a result, most practical solution methods do not guarantee optimality for problems with more than two decision stages \cite{gorissen2015practical}.

An additional layer of complexity arises when uncertainty sets depend on earlier decisions, commonly referred to as decision-dependent or endogenous uncertainty \cite{lappas2018robust,zhang2020unified}. While this modeling paradigm has gained increasing attention for practical applications \cite{haitao2025robust,tan2025adjustable}, exact solution approaches remain scarce and are typically limited to specialized structures \cite{chen2026robust,POSS2014836,nohadani2018optimization}. 

In the general context of optimization under uncertainty, several solvers, modeling tools, and software packages for robust optimization \cite{goh2011robust,dunning2016advances,wiebe2021romodel} even with endogenous uncertainty \cite{vayanos2020roc++}  as well as stochastic and distributionally robust optimization  \cite{chen2020robust} have been developed. Furthermore, for the special case of combinatorial three-stage  fortification games a general exact solution framework was presented in  \cite{leitner2021exact}. The  other mentioned solvers for robust optimization problems that can also cope with multistage settings, however, rely on approximation schemes---in particular decision rules---for adjustable variables and therefore cannot---in general---guarantee the optimality of the proposed solution.
\paragraph{Quantified Programming.}
Quantified constraint satisfaction has been studied since the mid-1990s \cite{gerber1995parametric} and has seen both theoretical and algorithmic advances \cite{mamoulis2004algorithms,gent2008solving,gao2020solving}, including early optimization variants \cite{benedetti2008quantified}.  Subramani coined the term quantified linear programs (QLPs) \cite{subramani2003analysis}, and later extended this framework to quantified integer programs \cite{subramani2004analyzing}.
This problem essentially is the decision version the Problem~\eqref{MRILP-Box}. Subsequent work incorporated a linear objective function and established QIPs as optimization problems \cite{Euro15}, leading to an early solver implementation \cite{YasolACG17}. Later extensions introduced a linear constraint system restricting the so called universally quantified variables (adversary variable in our notation), thereby enabling discrete uncertainty sets beyond simple bounds, including polyhedral and decision-dependent uncertainty \cite{hartisch2016quantified,hartisch2019mastering,DissMichael}. The expressive power of the arising modeling framework has been exploited to model multistage robust discrete optimization problems \cite{goerigk2021multistage} as well as Boolean and multilevel decision problems \cite{hartisch_et_al:LIPIcs.CP.2025.12}, and  decision-dependent robust optimization problems \cite{lefebvre2025solving,lefebvre2026branch}. From an algorithmic perspective, the solver presented in this paper follows a search-based approach, whereas an expansion-based solver for QIPs with discrete, polyhedral uncertainty was recently presented \cite{hartisch_et_al:LIPIcs.CP.2025.12}, which is also used in the computational experiments presented in this paper. 

While the present work naturally builds on the earlier version of our solver presented in the aforementioned conference paper \cite{YasolACG17}, it addresses a substantially broader class of robust optimization problem. In particular, it goes beyond discrete box uncertainty sets, introduces new solver capabilities, and provides an extensive computational evaluation on the newly supported problem classes. Furthermore, by adopting notation more commonly used in the robust optimization literature instead of the QIP notation employed in the earlier contributions, we aim to make the approach more accessible to the robust optimization community.

\paragraph{Quantified Boolean Formulas.}
QIPs are closely linked to quantified Boolean formulas (QBFs). Over the past two decades, QBF solving has seen substantial progress, including the development of both search-based and expansion-based solvers \cite{lonsing2017depqbf,janota2018towards,tentrup2019caqe}, as well as a wide range of specialized techniques such as clause learning, dependency schemes, and preprocessing methods \cite{lonsing2015enhancing,janota2015solving,heule2015clause,wimmer2017hqspre,lonsing2019qratpre}. Many of these ideas---most notably conflict-driven learning and non-chronological backtracking---inspire the algorithmic design of our solver, which adapts them to settings with linear constraints, integer variables, and optimization objectives.

\paragraph{Multilevel Optimization.}
Multilevel optimization \cite{migdalas2013multilevel} generalizes bilevel programming \cite{dempe2020bilevel} to settings with multiple hierarchical decision-makers. Despite its expressive power, most algorithmic research focuses on bilevel problems \cite{tahernejad2020branch,fischetti2017new}, and only few approaches explicitly address true multilevel settings with more than two levels \cite{bolusani2022framework,sato2021gradient}. Given the close conceptual connection between robust optimization and bilevel programming \cite{OJMO_2025__6__A2_0}, we examine our solvers behavior on bilevel instances in the experiments.

\subsection{Contribution and Outline}
We present a search-based approach for multistage robust discrete linear optimization problems with discrete, decision-dependent uncertainty in Section~\ref{sec::algo}. This stands in contrast to the commonly used reformulation, decomposition and approximation techniques used in multistage optimization. We introduce our solver Yasol in Section~\ref{sec::software}, and discuss enhanced techniques and heuristics used in our implementation and provide insights into its usage. In Section~\ref{sec::experiments}, we present computational results that evaluate our solver in several settings. The experiments demonstrate the modeling flexibility of \eqref{problem}, assess Yasol as a reference solver, quantify the impact of ignoring the multistage structure, and compare our method with decision-rule approximations. 
\section{The Algorithmic Framework\label{sec::algo}}

\subsection{Outline of the Algorithm}
At a high level, the solver performs a depth-first branch-and-bound search on the max--min search tree. The algorithmic framework is conceptually related to minimax search with \Alphabeta pruning \cite{KM75} and Quantified Conflict-Driven Clause Learning (QCDCL) \cite{zhang2002conflict,beyersdorff2021quantified}, but extends these paradigms to linear constraints, integer variables, and multistage optimization settings. 
The minimax-search maintains lower and upper bounds $\alpha$ and $\beta$ at each search node $v$, where the window between $\alpha$ and $\beta$ indicates which minimax-values may be of interest at $v$. If the minimax-value at $v$ is below $\alpha$, $v$ will be irrelevant for the max-player (the decision-maker). If it is above $\beta$, the node will be irrelevant for the min-player (the adversary). Thus, as soon as the window at a node $v$ collapses, the subtree below $v$ is cut off. 
The major advantage of \Alphabeta pruning is that, in the best case, it determines the minimax-value of the root with only a minimum number of search nodes: For a binary tree with uniform depth $d$, the best-case number of examined leafs is in $\lfloor 2^{d/2} \rfloor + \lceil 2^{d/2}\rceil - 1$.
This corresponds to the number of leafs for the decision-makers's strategy plus the number of leafs for the adversary's strategy minus that single leaf which both strategies have in common.  The principal variation---the path from that common leaf to the root---shows a the variable assignment arising when both agents make optimal decisions. Hence, detecting and verifying the principal variation is essential in order to obtain the objective value.
 
In Algorithm~\ref{Algo::ExtendedAlphaBeta} we sketch the function called for a new node. The search traverses through the tree recursively and fixes variables when going into depth or relaxes the fixations again when returning from a subtree. The decision level $d$ indicates the number of active decision nodes from the current search node $v$ to the root node, where implied  variables are not counted. We use $y$ to denote variables in this code, where it is noteworthy that it could be a decision variable $x$ or a adversary variable $\xi$. For simplification we write $\Domain(y)$ when referring to the bounded integer domain of variable $y$. In the following, we provide further insights into the most important aspects of this algorithm.

\begin{algorithm}[htbp!]
\footnotesize
\SetAlgoLined
\DontPrintSemicolon
\KwData{depth $d$, bounds $\alpha$ and $\beta$}
\KwResult{minimax value if $\alpha < \score < \beta$; bound on the minimax value otherwise}

\If{\label{Line::Extended1}current search node is a leaf}{
  \lIf{\label{Line::LeafCheckE}$\Ax\xvec + \Axi\xivec \not\leq \bvec,
$}{\Return $-\infty$}
  \lElseIf{\label{Line::LeafCheckA}$\Gx\xvec + \Gxi\xivec \not\leq \hvec$}{\Return $+\infty$}
  \lElse{\Return objective value}
}

\If{\label{Line::StartLP} LP relaxation is infeasible}{
      extract conflict clause, implied variable $y_{imp}$, and \targetlevel\; 
  \lFor{$k \coloneqq d$ \KwDownTo $\targetlevel$}{$\finished[k] \coloneqq \true$}
  $\implied[\targetlevel] \coloneqq y_{imp}$\;
   \Return $-\infty$\label{Line::ReturnLP}\;
}

\lIf{current search node is a max node}{$\score \coloneqq -\infty$}
\lElse{$\score \coloneqq +\infty$}
$\finished[d] \coloneqq \false$\; 
$\implied[d] \coloneqq \none$\;

determine branching variable $y$ and the branching order \polarity\label{Line::SelectVar} \;

\For{$i \in \Domain(y)$}{
  \lIf{$y$ is adversary variable and $\polarity[i]$ is  inadmissible\label{Line::AllCheck}}{\Continue}
\lElse{ \Assign{$y,\polarity[i]$}\label{Line::Assign}}
  \Repeat{\label{line::ImplicationLoop}$\implied[d] == none$}{
    \If{\label{Line::ignoreFinished} \textnormal{\finished[d]}}{
      \Unassign{$x_v$}\;
      \lIf{$y$ is decision variable}{\Return $\score$}
      \lElse{\Return $-\infty$}
    }
	\If{\label{Line::Implied}$\implied[d] \neq none$}{
	\label{Line::AssignImplied} assign $\implied[d]$\;
	$\implied[d] \coloneqq none$\;}
    \eIf{\label{Line::Propagate}  propagation produces no conflict}{
      \eIf{current search node is max node}{
        $\valuevar \coloneqq \textnormal{\texttt{Search}}(d+1,\max\{\alpha,\score\},\beta)$ \label{Line::yAlpha1} \;
        \lIf{$\valuevar>\score$}{$\score \coloneqq \valuevar$}
        \lIf{\label{Line::PruningB}$\score\geq\beta$}{  \Unassign{$y$}; \Return $\score$;}
      }{
         $\valuevar \coloneqq \textnormal{\texttt{Search}}(d+1,\alpha,\min\{\beta,\score\})$\label{Line::yAlpha2} \;
        \lIf{$\valuevar<\score$}{$\score \coloneqq \valuevar$}
        \lIf{\label{Line::PruningA}$\score\leq\alpha$}{  \Unassign{$y$}; \Return $\score$;}
      }
    }{
     \label{Line::conflictStart} extract conflict clause, implied variable $y_{imp}$, and \targetlevel\; 
  \lFor{\label{Line::markFinished} $k \coloneqq d$ \KwDownTo $\targetlevel$}{$\finished[k] \coloneqq \true$}
  $\implied[\targetlevel] \coloneqq y_{imp}$\;
   \Unassign{$y$} \label{Line::UnassignProp} \;
      \lIf{$y$ is decision variable}{\Return $\score$}
      \lElse{\label{Line::conflictEnd} \Return $-\infty$}
    }
  }

  \Unassign{$y$}\;

}
\label{Line::EndLoop}
\lIf{\label{Line::Extended10} $\score=+\infty$ and $\Ax\xvec + \Axi\xivec \leq \bvec,
$ has no solution}{\label{Line::ExistInfeasible}\Return $-\infty$}
\lElse{\Return $\score$\label{Line::NormalReturn}}
\caption{\texttt{Search}($d,\alpha,\beta$) call to process node in minimax tree.\label{Algo::ExtendedAlphaBeta}}
\end{algorithm}

\subsection{Selecting Branching Variables and the Branching Order}
The efficiency of the sketched algorithm heavily depends on the search node ordering. Branching proceeds stage by stage: as long as variables remain unassigned in the current stage, only variables of that stage are considered for branching (Line~\ref{Line::SelectVar}). For branching variable $y$, the branching order \polarity\xspace determines the order in which the candidate values from its domain $\Domain(y)$ are considered. Variable and assignment selection within a stage is guided by data collected during the search as well as established branching rules from integer linear programming~\cite{achterberg2005branching}. Furthermore, having reasonable knowledge of which adversary variable assignments are particularly crucial can massively boost the search process. Several heuristics exist to analyze and find such promising actions in a game-tree search environment \cite{Killer,WinandsHistory,Schaeffer,Plaat}. 

\subsection{Propagation}
After the branching variable is assigned, a propagation procedure is invoked (Line~\ref{Line::Propagate}). This procedure performs implications such as unit propagation \cite{beyersdorff2021quantified} and bound propagation \cite{AchterbergDiss}. In particular, the adversary constraint system is used for implications on adversary variables and the decision-maker constraint system is used for implications on decision variables. 
It is worth noting that in case of a decision-dependent uncertainty set, variables whose domains collapse to a single value during propagation should be assigned explicitly only if they belong to the same stage as the branching variable $y$. Assigning variables from later stages prematurely may otherwise interfere with subsequent admissibility checks.

\subsection{Checking admissibility\label{sec::legality}}
In the presence of an adversary constraint system, the algorithm must deal with the admissibility of assignments. However, checking the admissibility of a variable assignment explicitly requires checking the feasibilty of an IP (see Definition~\ref{def::legal}), which is computationally costly. Therefore, instead of checking the admissibility of every variable assignment explicitly, we show that it suffices to ensure that adversary variable assignments have to be set in an admissible manner at any time. Inadmissible decision variable assignments, on the other hand, will be detected later on during the search, allowing us to avoid solving the decision-maker IP over and over again. We note that this procedure also could be done the other way around by explicitly ensuring that decision variables are admissible. However, in many robust optimization settings, the adversary constraint system tends to be much smaller than the decision-maker constraint system: the description of the uncertainty set in many discrete applications is rather small, e.g., only containing budget constraints \cite{bertsimas2004price}, cardinality constraints \cite{goerigk2021multistage}, or simple decision-dependent structures \cite{baggio2021multilevel}.

The key reason why Algorithm~\ref{Algo::ExtendedAlphaBeta} only needs to explicitly enforce the admissibility of adversary assignments to solve Problem~\eqref{problem} is the following: if both constraint systems are violated, then this must originate from an inadmissible decision-maker assignment. Otherwise, it would require an inadmissible adversary assignment, which the algorithm has already excluded.
Due to Assumption~\ref{ass::1}, we know that at least the adversary constraint system still has a solution at the root node. In Algorithm \ref{Algo::ExtendedAlphaBeta} the admissibility of adversary variable assignments is detected in Line \ref{Line::AllCheck}, skipping the respective subtree. Hence, only admissible adversary variable assignments can change the score of the current subtree. We now need to ensure that the search returns the correct value even after making an inadmissible decision-maker  assignment. To this end, let us first deal with the case that at an adversary search node $v$ no admissible assignments exist. As \texttt{score} is initialized to  $+\infty$ and all iterations of the for-loop immediately end in Line~\ref{Line::AllCheck}, \texttt{score} is still $+\infty$ at the end of the loop in Line~\ref{Line::EndLoop}. There are two potential cases:

\begin{itemize}
    \item[1.] The current adversary node was reached from an adversary parent node. This, however, must have been an inadmissible variable assignment, as no admissible assignment exists in the current node. Hence, with Line~\ref{Line::AllCheck} in Algorithm~\ref{Algo::ExtendedAlphaBeta}, the search cannot not have traversed from an adversary parent node. 
    \item[2.] The adversary node is reached from a decision-maker parent node. There are two further cases
    \begin{itemize}
        \item[a)] There still exists a solution of the decision-maker constraint system within $\Domainx\times\Domainxi$, i.e., the assignment of the decision variable that lead to the current node was admissible. Hence, the current adversary node should be marked as loss for the adversary. As in this situation Line \ref{Line::NormalReturn} is reached, the search returns the correct score of the subtree.
\item[b)] There exists no solution of the decision-maker constraint system within $\Domainx\times\Domainxi$, i.e., the assignment of the decision variable which lead to the current node was inadmissible. This will be detected in Line \ref{Line::ExistInfeasible}, resulting in the return of the value $-\infty$, indicating the inadmissible action. In particular, we know, that the violation of the adversary constraint system arose by an inadmissible assignment by the decision-maker, as the adversary constraint system was still intact at the start of the decision-maker's stage, due to Line~\ref{Line::AllCheck} and Assumption~\ref{ass::1}. 
    \end{itemize}
        \end{itemize}

\begin{corollary}
Whenever Algorithm~\ref{Algo::ExtendedAlphaBeta} first enters a node corresponding to a new decision-maker stage, it is guaranteed that the adversary constraint system still has a solution within $\Domainx\times\Domainxi$.
\end{corollary}

Now assume that we are at a decision-maker search node and the search has made an inadmissible assignment in Line~\ref{Line::Assign}. If this is detected during propagation, the inadmissible assignment will be undone immediately (Line~\ref{Line::UnassignProp}). If it is not detected immediately, the value returned in Line~\ref{Line::yAlpha1} from the respective child node should be $-\infty$. We show this in the following theorem.
\begin{theorem}
Algorithm~\ref{Algo::ExtendedAlphaBeta} returns $-\infty$ for any search node reached via an inadmissible decision-maker assignment.
\end{theorem}

\begin{proof}

We show this by induction. 
Assume that the next search node following an inadmissible decision-maker assignment is a leaf. Then the value returned will be $-\infty$ as  Line~\ref{Line::LeafCheckE} is reached. In particular, regardless of the status of the universal constraint system, $-\infty$ is returned. Now assume the statement is true for any node at depth $d$. Consider a node at depth $d-1$ reached via an inadmissible decision-maker assignment. If the violation of the decision-maker constraint system is not detected while solving the relaxation in Line~\ref{Line::StartLP} (resulting in the return of $-\infty$ in Line~\ref{Line::ReturnLP}) the next variable will be assigned. There are two cases:
\begin{itemize}
\item[1.] This node corresponds to a decision-maker assignment. As this node was reached via an inadmissible assignment, any assignment of yet unassigned decision variables will by definition also be inadmissible. Hence, by induction, any recursive call in Line~\ref{Line::yAlpha1} will return $-\infty$ and therefore the return value will remain $-\infty$. 
\item[2.] This node corresponds to an adversary assignment.
\begin{itemize}
\item[a)] If there are no admissible assignment for the adversary variable, the condition in Line~\ref{Line::Extended10} is satisfied and $-\infty$ is returned. 
\item[b)] If there are admissible adversary assignments, the search will proceed to nodes at depth $k$ via Line~\ref{Line::yAlpha2} and the statement follows by induction.\qed
\end{itemize} 
\end{itemize}
\end{proof}

Therefore, it suffices to only explicitly check the admissibility of adversary variable assignments, while the feasibility of the decision-maker constraint system only has to be checked explicitly in rare cases (Line~\ref{Line::Extended10}).

Still, as stated before, checking the admissibility of adversarial variable assignments requires the solution of the feasibility problem on the adversarial constraint system, i.e., and IP has to be solved. In some use cases, however, we additionally can benefit from a simple structure of the underlying uncertainty set, where it is rather easy to see that a specific realization of the adversarial variable assignment is inadmissible. We call such instances \textit{simply restricted}.
\begin{definition}[Simply Restricted Instance]\label{Def::SimplyRestrictedQIPID}~\\
An \eqref{problem} instance is called \emph{simply restricted}, if for any adversary variable $j\in\{1,\ldots,k_t\}$ in any stage $t\in\{1,\ldots,T\}$ and any possible prefix $\xvech[0],\xivech[1],\xvech[1],\ldots,\xvech[t-1]$, $\hat{\xi}^t_1,\ldots,\hat{\xi}^t_{j-1}$, the following holds: 
\begin{align*}
&\tilde{\xi}^t_j \not\in \Legalxi^t (\xvech[0],\xivech[1],\xvech[1],\ldots,\xvech[t-1],\hat{\xi}^t_1,\ldots,\hat{\xi}^t_{j-1}) \Rightarrow \nonumber\\
&
\exists \, r \in \{1,\ldots,q\}:
\min_{\substack{ (\xvec,\xivec)\in \Domainx \times \Domainxi:\\ 
\xvec[p] = \xvech[p] \ \forall p \in \{0,\ldots,t-1\} \\
\xivec[p] = \xivech[p] \ \forall  p \in \{1,\ldots,t-1\} \\
\xi^t_f=\hat{\xi}^t_f\ \forall  j \in \{1,\ldots,j-1\}  \\
\xi^t_j=\tilde{\xi}^t_j}} (\Gx\xvec)_r+(\Gxi\xivec)_r > \hvec_r
\label{Condition:EasyDetectIllegalAllMove}
\end{align*}
\end{definition}
In other words: the inadmissibility of  an adversary assignment $\tilde{\xi}^t_j$ can be detected by finding a single adversary constraint $r$ that can no longer be satisfied, even in the best case. Hence, it suffices to check the adversary constraints \textit{separately} to verify feasibility of the \textit{full} adversary constraint system and, consequently, the admissibility of the assignment $\tilde{\xi}^t_j$. We note that no automatic detection of simply restricted instances is implemented in our solver and the responsibility to detect it is handed over to the user. However, some structural features where already analyzed in \cite{DissMichael}.

\subsection{Pruning Mechanisms}
Three types of pruning mechanisms are employed during the search. First, alpha-beta pruning eliminates subtrees of the search space that cannot improve the current best objective bounds by maintaining lower and upper bounds on attainable values at each search node~\cite{KM75} (see Lines \ref{Line::PruningB} and \ref{Line::PruningA}). Second, if a variable assignment and subsequent propagation produces a conflict, a conflict constraint is generated, summarizing the cause of infeasibility to prevent the search from revisiting the same conflicting partial assignment again (see Lines~\ref{Line::conflictStart}--\ref{Line::conflictEnd}). To enable efficient backtracking, such constraints are transformed into compact clause-like constraints in Yasol, allowing the application of non-chronological backjumping similar to QCDCL and modern SAT solvers \cite{Qube,lonsing2015enhancing}. Further details will be provided in Section~\ref{sec::conflicts}. Third, a subtree is pruned if the relaxation is infeasible (see Line~\ref{Line::StartLP}--\ref{Line::ReturnLP}), in which case a similar backtracking mechanism is invoked based on a feasibility cut extracted from the dual LP using classical Benders/Farkas reasoning.

\subsection{Relaxations}
At each search node, the decision-maker constraint system is relaxed to a linear program (LP), which is used to bound the objective value, detect infeasibility, and guide the search (see Lines~\ref{Line::StartLP}--\ref{Line::ReturnLP}). Besides relaxing integrality, one approach is to treat adversary variables as decision-maker variables. This relaxation, however, is often too weak, as it ignores the worst-case nature of adversary variables. In particular, allowing the decision-maker to assign adversary variables disregards the min--max semantics and typically yields overly optimistic bounds.

Whenever possible, we therefore employ stronger relaxations in which adversary variables are fixed to assignments that are guaranteed to be admissible; so called scenarios. The underlying intuition is that requiring the decision-maker's decisions to remain feasible for a particular adversarial assignment yields a more informative relaxation than allowing the decision-maker to choose the adversary's actions. Such a relaxation can be strengthened further by considering several adversarial scenarios simultaneously \cite{DissMichael,hartisch2021adaptive}. This corresponds to expanding a small subset of scenarios while relaxing integrality: the relaxation must identify a continuous decision strategy that satisfies the decision-maker constraint system for every included scenario, while respecting the information structure shared across scenarios. Consequently, both the choice and the number of scenarios are important for the strength and computational cost of the relaxation. To guide scenario selection, \Yasol collects information on universal variables during the search using the VSIDS heuristic \cite{moskewicz2001chaff}, the killer heuristic \cite{Killer}, and occurrence counts for scenarios and subscenarios.

We note that these relaxations are not always applicable in presence of general decision-dependent uncertainty, where the admissibility of adversary assignments depend on the assignment of decision variable. 
In this case, we can fall back on treating adversary variables as continuous as well as controlled by the decision-maker, except for those that do not appear in the adversary constraint system, which still can be fixed to some admissible scenario. Additional theoretical results on different relaxations have been obtained in \cite{DissMichael} and an experimental evaluation of the relaxation techniques are provided in \cite{hartisch2021adaptive}.

\subsection{Conflict Analysis and Backpropagation \label{sec::conflicts}}

Standard minimax search explores the search tree in chronological order and may therefore spend substantial effort in branches that have already become irrelevant. To avoid this, our implementation combines conflict analysis, clause learning, and non-chronological backtracking, following ideas from QBF solving \cite{zhang2002conflict,lonsing2013efficient,beyersdorff2021quantified} and mixed-integer programming \cite{achterberg2007conflict,witzig2021computational}.

A conflict in the decision-maker constraint system can arise in two ways. First, bound propagation (see Line~\ref{Line::Propagate}) may detect that, under the current variable assignments and bounds, a constraint can no longer be satisfied by any completion of the remaining variables. Second, the LP relaxation may be infeasible (see Line~\ref{Line::StartLP}). In the latter case, a Farkas certificate obtained from the dual system provides a linear explanation of infeasibility, analogously to the feasibility cuts used in Benders decomposition \cite{benders1962partitioning}. In either case, the result is a linear conflict constraint that certifies that the current combination of assignments cannot lead to a feasible solution.

To exploit this information efficiently during search, the conflict constraint is transformed into a clause over binary variables. This is natural in Yasol, since general integer variables are represented internally by binary variables. The basic idea is to identify assignments that jointly cause the conflict and to exclude their simultaneous occurrence. Hence, if the conflicting assignment contains $y_i=1$, the corresponding clause contains the literal $\neg y_i$, whereas an assignment $y_i=0$ contributes the literal $y_i$. The resulting clause is generally weaker than the original linear conflict constraint because the numerical coefficient information is discarded. However, it can be stored and propagated very efficiently using the two-watched-literal scheme \cite{moskewicz2001chaff}.

If one of the assignments contributing to the conflict was itself implied by propagation, the corresponding implied literal is not retained directly. Instead, the conflict clause is resolved with the reason for that implication, thereby replacing the implied variable by the assignments that caused it. This process is repeated recursively until a suitable learned clause is obtained. In this way, the final clause explains the conflict in terms of earlier branching decisions rather than intermediate propagated assignments.

The learned clause is then used to determine a backtracking level. Rather than returning only to the immediately preceding decision level, conflict analysis identifies the earliest level at which the conflict is imminent. The search can therefore skip intermediate levels that do not affect the conflict. In Algorithm~\ref{Algo::ExtendedAlphaBeta}, these levels are marked as finished via \finished\xspace and are closed when the recursion reaches them. Once the \targetlevel\xspace is reached, the learned implication is enforced by setting \implied\xspace accordingly, thereby preventing the search from reproducing the same conflicting assignment pattern.

\paragraph{Example:}
Assume that conflict analysis yields the linear constraint
\[
3x_1 + 5x_7 + 6x_9 - 3x_{15} - 3x_{21} \ge 6
\]
and that the current search path contains
\[
x_1=1,\qquad x_7=0,\qquad x_9=0,\qquad x_{15}=1,\qquad x_{21}=0.
\]
Under these assignments, the left-hand side cannot reach the required value unless at least one of the assignments $x_7=0$, $x_9=0$, or $x_{15}=1$ is changed. Ignoring the coefficients therefore yields the conflict clause
\[
x_7 \lor x_9 \lor \neg x_{15}.
\]

Suppose, however, that $x_9=0$ was not chosen by branching but implied by another constraint, for instance
\[
x_2+x_3+x_9\ge1,
\]
after $x_2=0$ and $x_3=0$ had been fixed earlier in the search. Resolving the conflict clause with this implication replaces the implied literal involving $x_9$ by its reason and yields
\[
x_2 \lor x_3 \lor x_7 \lor \neg x_{15}.
\]
The resulting conflict clause now expresses the conflict solely in terms of earlier active assignments. Assume that the search branches on variables in increasing order of their indices. In particular, assume $x_{15}$ was assigned in level $a$ and $x_7$ was assigned in level $b<a$. Then, already at level $b$, after assigning $x_7=0$, $x_{15}=0$ can be implied to avoid the conflict. Hence, the \targetlevel\xspace is set to $b$ and \implied\xspace is set to $x_{15}$. 

\section{Software Framework\label{sec::software}}
\subsection{Implementation Details and Assumptions}
Yasol is a search-based solver for \eqref{problem} implemented in \texttt{C++}. Its solution process combines ideas from game-tree search, quantified Boolean Formula, and mixed-integer programming and generally follows Algorithm~\ref{Algo::ExtendedAlphaBeta}. Internally, all discrete variables are represented using a binary encoding. This representation forms the basis for the search process and enables the adaptation of highly efficient SAT and QBF solving techniques, including the 2-literal watching scheme \cite{moskewicz2001chaff}, variable implications, conflict analysis, and non-chronological backtracking \cite{zhang2002conflict,lonsing2013efficient,beyersdorff2021quantified}, while still maintaining and leveraging the integer information where possible. The option to binarize adversary variables has only become possible with the introduction of the adversary constraint system. 

 
Based on this representation, the search process implicitly constructs an optimal strategy by exploring the minimax tree and evaluating partial assignments. Nodes corresponding to decision-maker variables are evaluated by taking the maximum over their successors, while nodes corresponding to adversary variables are evaluated by taking the minimum. We note that, reflecting the solver's origins in QBF solving, the code base uses terminology and variable names based on existential and universal quantification. This notation differs from the decision-maker/adversary terminology adopted in this paper, but also reflects a genuine asymmetry in how the two types of variables are treated algorithmically.

To mitigate heavy-tailed runtime behavior, the solver employs heuristic and randomized restarts \cite{gomes1998boosting}, which renders the overall search process non-deterministic. The search terminates once all branches are either explicitly explored or pruned. At this point, optimality follows from the maintained objective bounds and the correctness of the recursive evaluation. Decisions along branches that determine the objective value are recorded, allowing the solver to extract the sequence of variable assignments corresponding to the reported objective value.

As outlined in the previous section, at each search node, a relaxation is maintained that approximates the remaining problem. These LP relaxations serve two purposes: bounding the objective value and detecting infeasibility early. To this end, Yasol requires an LP solver backend and currently supports linking against IBM ILOG CPLEX, HiGHS \cite{huangfu2018parallelizing}, or the COIN-OR tools \cite{lougee2003common}. The relaxations are tailored to the adversarial structure: adversary variables are handled in a worst-case manner \cite{hartisch2021adaptive}, while decision-maker variables are optimized, yielding bounds consistent with the max--min structure. Furthermore, relaxations are constructed incrementally---reusing bases and solver states across neighboring nodes whenever possible---and managed lazily: only constraints that are relevant for the current search node are activated, which avoids repeatedly solving large and weak LPs. Standard cutting planes, such as Gomory mixed-integer cuts and cover cuts, are generated to strengthen relaxation bounds but are not permanently added to the primary constraint database. We note that the external solvers IP capabilities are only invoked if an instance is not \emph{simply restricted} and hence admissibility must be verified by finding an integer solution to the adversary constraint system. 
 
In addition to the techniques outlined in the previous section, Yasol incorporates several preprocessing techniques inspired by QBF and MIP solving \cite{beyersdorff2021quantified,achterberg2009scip,DissMichael}. These include constraint simplifications, universal reduction, variable fixings, and the exploitation of dominance relations. These techniques had to be adapted with particular care to \eqref{problem} setting, as the presence of linear constraints---as opposed to clauses---as well as interactions between variables from different stages and agents can fundamentally alter the validity of standard reasoning steps \cite{DissMichael}. Similarly, we can utilize the well-known concept of  monotonicity \cite{cadoli2002algorithm} for both adversary as well as decision-maker variables, that allows us to omit certain subtrees a priori since solving the subtree of its sibling is guaranteed to yield the desired minimax value. Furthermore, a solution-based pruning technique has been implemented, where partial strategies are reused whenever alternative adversary assignments remain feasible and do not worsen the objective value~\cite{hartisch2019novel}.

\subsection{Usability and Software Features}

\paragraph{Availability and Installation\label{sec:design}}
Yasol is open-source and available on GitHub under the MIT License\footnote{\url{https://github.com/MichaelHartisch/Yasol}}. The solver can be built using CMake and we provide a configuration script that automates setup, including the optional download of supported open-source external solvers. The build process produces both a standalone executable and library files for integration into external projects.

\paragraph{Supported Problem Classes}
The solver is capable of solving instances of type \eqref{problem} as introduced in Section~\ref{sec::problem}. In addition, continuous variables are supported in a final decision-maker stage, allowing mixed-integer recourse. This is possible as the search as outline in Algorithm~\ref{Algo::ExtendedAlphaBeta} branches on binary variables and assigns continuous variables via the final call to the LP solver at leaves.

\paragraph{Input Format}
When used as a standalone executable, Yasol expects instances to be provided in the QLP file format, which extends the CPLEX LP file format\footnote{\url{http://lpsolve.sourceforge.net/5.1/CPLEX-format.htm} (accessed October~15,~2026)}. The QLP format introduces the keywords \texttt{EXISTS}, \texttt{ALL}, and \texttt{ORDER}. This naming scheme has emerged from the solver's origin as QIP solver. The format requires the user to state all decision variables $\xvec$ below  \texttt{EXISTS} and all adversary variable $\xivec$ below  \texttt{ALL}. Below the \texttt{ORDER} keyword, the user has to specify the order of the variables, i.e.,  $\xvec[0],\xivec[1],\xvec[1],\ldots,\xivec[T],\xvec[T]$.  To support constraints on adversary variables, the format additionally provides the keyword \texttt{UNCERTAINTY SUBJECT TO}.

\paragraph{Output}
The solver reports the optimal objective value, which corresponds to the worst-case outcome of an optimal strategy. In addition to the optimal first-stage assignment, the sequence of variable assignments attaining this optimal value is returned; the full strategy tree is not provided. Solutions can be written to an XML-formatted output file, which also stores the best incumbent solution if computation is terminated early.

\paragraph{Library Interface}
In addition to the standalone executable, Yasol can be embedded into external software projects. The library interface allows users to construct problem instances, invoke the solver, and access solution information directly. Example code demonstrating the use of the library interface is provided in the repository.

\paragraph{Parameters and Options}
The solver provides a range of configuration options that allow users to control its behavior. A summary of available options can be obtained by invoking the solver with the help flag. Key options include the specification of a time limit, control over the generation of solution output, activation or deactivation of cut generation and presolving techniques, and the ability to indicate whether an instance is simply restricted.

\section{Computational Experiments\label{sec::experiments}}
We examined three different problem classes and data sets from different research areas to evaluate the capabilities of our solver. All experiments were carried out on a server equipped with an AMD EPYC
505 9474F 48-Core processor (3.60 GHz) with 256 GB RAM. Tasks were parallelized via GNU Parallel \cite{tange2011gnu} and at most 48 tasks were carried out at the same time. All calls to any solver were given a time limit of 30 minutes per instance.

We used \Yasol version 4.6 (git tag v4.6) which was compiled with CPLEX 22.1.1 \cite{cplex}. Furthermore, whenever possible, we compared against the performance of the expansion-based solver \Equips \cite{hartisch_et_al:LIPIcs.CP.2025.12} using Gurobi 12.0 as IP solver. This solver is not able to deal with decision-dependent uncertainty, but can solve problems of type~\eqref{problem} with a polyhedral uncertainty set (i.e., where $\Gx=0$).
\subsection{Multilevel Critical Node Problem}
The multilevel critical node problem is a protection-interdiction problem \cite{baggio2021multilevel}. It is a graph optimization problem involving two agents: an attacker, who selects nodes to infect, and a defender, whose objective is to maximize the number of nodes that remain uninfected. First, the defender may vaccinate a set of nodes. The attacker then selects nodes to infect, after which the defender may protect an additional set of nodes. This problem can be modeled either with a budgeted uncertainty set---where infecting vaccinated nodes is allowed but not enforced---or with a decision-dependent uncertainty set---where infecting vaccinated nodes is explicitly forbidden. A detailed description of the different models of the same problem is provided in Appendix~\ref{App::MCN}. We refer to runs of \Yasol on instances modeled with budgeted uncertainty set as \MCNBasic and we refer to runs of \Yasol on instance with decision-dependent uncertainty as \MCNDD.

We compared against the baseline column- and row-generation approach~\cite{baggio2021multilevel}, referred to as \MCNBaggio, for which we used the python tool made available by the authors\footnote{https://github.com/mxmmargarida/Critical-Node-Problem}.
As their implementation is based on Python version 2.7, this code had to be linked with the CPLEX 12.9. It is worth noting that the authors also provided more sophisticated and tailored solution approaches, but it is our intention to showcase the potential \Yasol has as a baseline approach. Furthermore, we compared against the performance of  \Equips on instances with budgeted uncertainty. 

We ran experiments on all instances provided in the aforementioned repository, containing randomly generated trees and undirected graphs, as well as on public graph instances (where we removed three instances for which no solution is made available). We report the number of solved instances in Table~\ref{tab::MCNExperiments} and show the performance profiles \cite{dolan2002benchmarking} on the three testsets in Figure~\ref{fig::MCNExperiments}. The performance profiles compare the runtimes of the solvers relative to the best solver on each instance. For a performance ratio $\tau \geq 1$, the value on the $y$-axis gives the fraction of instances for which a solver is at most a factor $\tau$ slower than the best observed runtime. Thus, the value at $\tau=1$ indicates how often a solver is the fastest, while the height reached for large $\tau$ reflects the fraction of instances solved successfully. 
\begin{table}[htbp]
\footnotesize
\centering
\caption{Summary of the computational results showing the number of solved instances by each solver on the different test sets. Bold numbers indicate best values. \label{tab::MCNExperiments}}
\begin{tabular}{rrrrr}
\toprule
&\MCNEQUIPS&\MCNBaggio&\MCNDD&\MCNBasic \\
\midrule
tree (600) &536&460&536&\textbf{600}\\
rndgraph (1560) &1292&1292&\textbf{1397}&1354\\
public\_repository (137)&111&\textbf{120}&106&105\\
total&1939&1872&2039&\textbf{2059}\\\bottomrule
\end{tabular}
\end{table}
\begin{figure}[t!]
    \centering
    \begin{subfigure}[t]{0.5\textwidth}
        \centering
        \includegraphics[width=\textwidth]{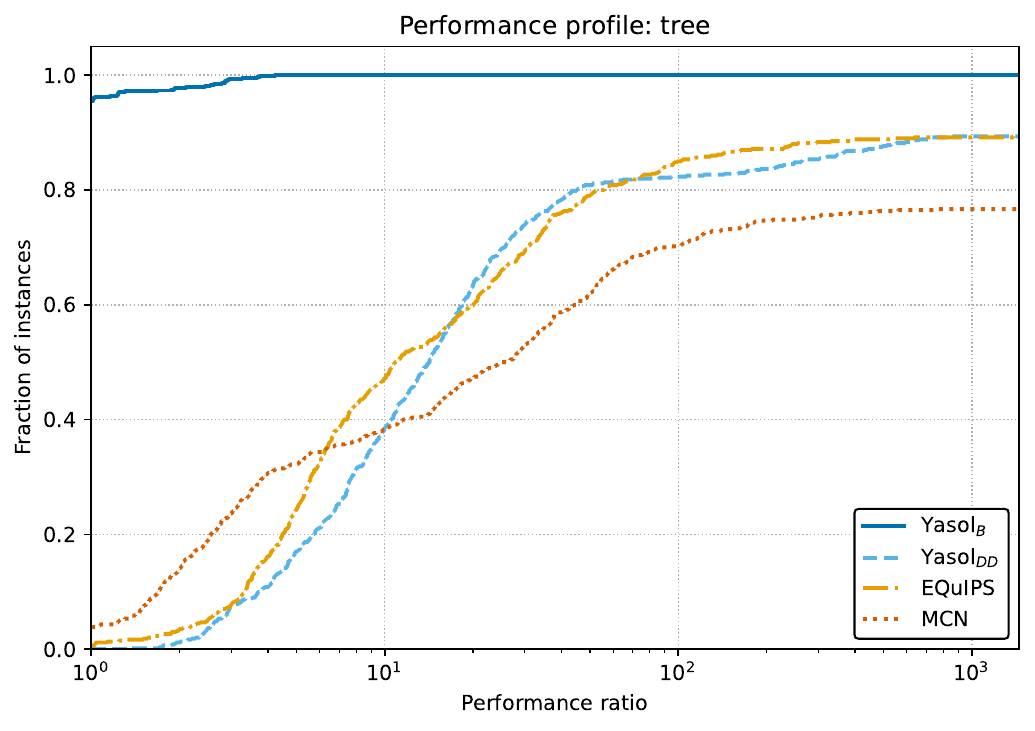}
        \caption{testset ``tree''}
    \end{subfigure}%
    ~ 
    \begin{subfigure}[t]{0.5\textwidth}
        \centering
        \includegraphics[width=\textwidth]{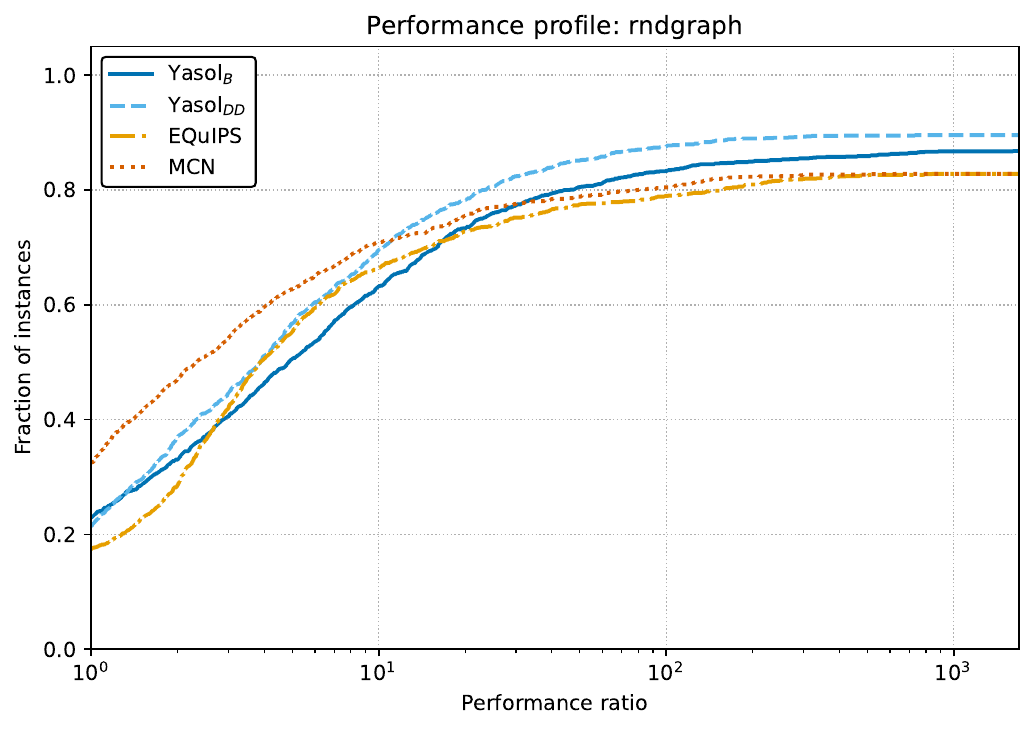}
        \caption{testset ``rndgraph''}
    \end{subfigure}
\\
    \begin{subfigure}[t]{0.5\textwidth}
        \centering
        \includegraphics[width=\textwidth]{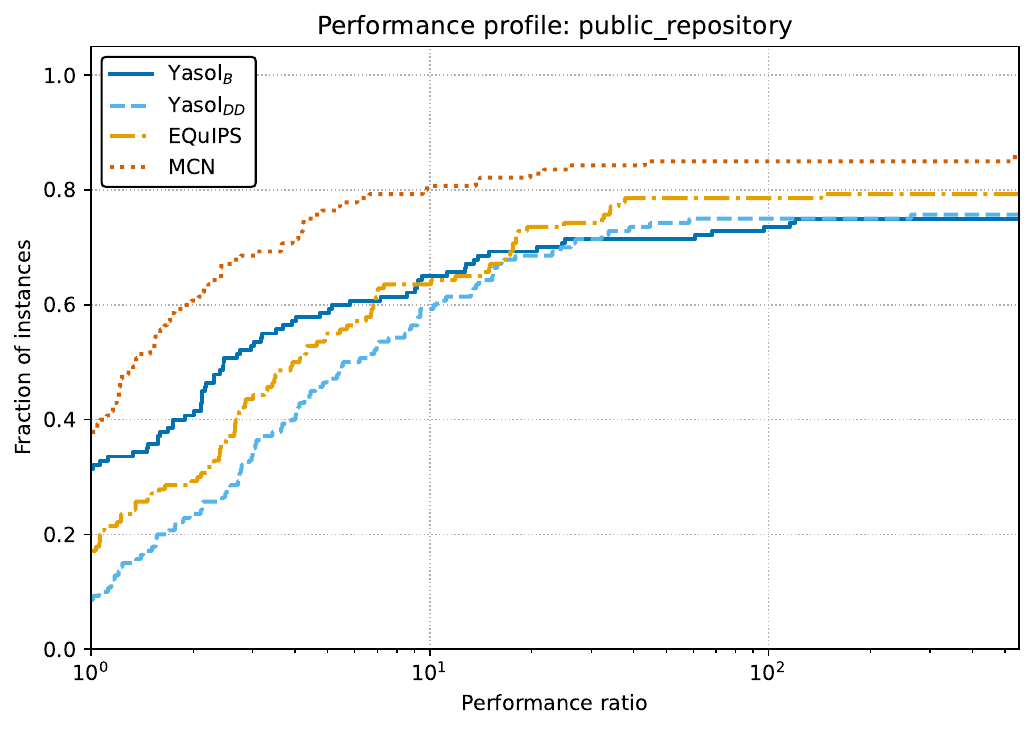}
        \caption{testset ``public\_repository''}
    \end{subfigure}
    \caption{Performance profiles of the four approaches on multilevel critical node instances.\label{fig::MCNExperiments}}
\end{figure}
These experiments show that solver performance depends strongly on the characteristics of underlying graph. In particular, no single solver performs best across all test sets; instead, each test set favors a different approach. Overall, \Yasol performs best on instances modeled with polyhedral uncertainty, closely followed by \Yasol on instances with decision-dependent uncertainty. It is worth highlighting, the extraordinary performance of \MCNBasic on the tree test set, where every instance is solved within less than $5$ minutes, while other solvers do not solve them within half an hour. 
\subsection{Bilevel Instances \label{sec::exp::bilevel}}
Given the close connection of robust optimization and bilevel optimization \cite{OJMO_2025__6__A2_0}, we furthermore consider the recently introduced Bilevel Optimization Benchmark Instance Library (BOBILib) \cite{Thuerauf_et_al:2026}. Several of the benchmark instances provided there can be transformed into instances of type~\eqref{problem}. Most importantly, only instances with objective function alignment of $-1$ can be considered, as they represent min-max problems. Furthermore, we only consider instances with pure integer variables. Consequently, we focused on instances provided in the interdiction benchmark of BOBILib. We removed all instances of type \texttt{generalized}, as they contain continuous lower level variables. Furthermore, we had to remove five instances from the \texttt{inter-assig} folder, as their objective function alignment was not $-1$. We also did not consider any \texttt{or} and \texttt{clique} instances as none of them were solved by any solver within the time limit in preliminary tests.

We implemented a tool\footnote{https://github.com/MichaelHartisch/mibs2qlp} for transforming instances in the MibS input file format format \cite{Thuerauf_et_al:2026} into instances in the QLP file format. This tool uses the notion of the leader being the decision-maker and the follower being the adversary, which we would claim to be the intrinsic interpretation. These instances are referred to be of type follower-adversary (FA). However, we are also able to invert this perspective, modeling the leader decisions as adversary variables and view  the follower as the decision-maker (referred to as leader-adversary (LA)). In particular, changing the responsibilities for the variables has an impact on the problem's properties, as the adversary is no longer in charge of the decision-dependent constraint system of the follower, but is in charge of often only few constraints. In particular, all FA-instances have a decision-dependent uncertainty set, while LA-instances have a polyhedral uncertainty set (details can be found in Appendix~\ref{app::bilevel::switch}). Additionally, this allows us to also run the expansion-based solver \Equips on LA-instances. Furthermore, we use the open-source solver\footnote{https://github.com/coin-or/MibS} \Mibs (version 1.2.2)  \cite{denegre2009branch,tahernejad2020branch}, which we linked against CPLEX 22.1.1.

 In Table~\ref{tab::bilevel::numsolved}, the number of solved instances is shown for the four considered approaches, which are \Mibs on the original bilevel instances, \Yasol on the FA-instances, as well as \Yasol and \Equips on the LA-instances.
 \begin{table}[htbp]
 \centering
 \footnotesize
 \caption{Number of solved instances of the interdiction benchmark set. Number of instances per set are indicated in the first column. Bold numbers indicate best values.\label{tab::bilevel::numsolved}}
\begin{tabular}{rrrrrr}
\toprule
 &\Mibs &\Yasol (FA)&\Yasol (LA)&\Equips (LA) \\\midrule	 
\texttt{bcpins} (80) & 21  & 20 & 45 & \textbf{80}\\    
\texttt{cclw} (50)    &  14  & 2 & \textbf{15} & 8\\    
\texttt{imkp} (144)   &   21    & 23 & \textbf{44} & 28\\    
\texttt{inter-assig} (19)
& 19  & 19 & 19 & 19 \\    
\texttt{inter-fire} (72) & 52 & 27 & \textbf{66} & 64 \\    
\texttt{inter-kp} (99)  & \textbf{41} & 39 & \textbf{41}& 36\\    
\texttt{kp} (450) &\textbf{75}& 63 & 63 & 64\\    
\texttt{plusbcpins} (80)& 20  & 14 & 30 & \textbf{45} \\\bottomrule
\end{tabular}
  \end{table}
Overall \Yasol performs much better on LA-instances than on FA-instances. This is expected, since the complexity of the uncertainty set has a substantial impact on the behavior of \Yasol; see Section~\ref{sec::legality}. Moreover, all LA instances can be tagged as \textit{simply restricted}, which further reduces the overhead because no external IP solver has to be called. On the LA-instances, only for instances of type \texttt{kp}, \Yasol solves less instances than \Mibs. The solver \Equips performs extraordinarily well on (\texttt{plus})\texttt{bcpins} instances. 

We also examine how the different approaches perform with respect to the difficulty classification of the benchmark instances. In the BOBILib, an instance is labeled as easy if all tested solvers are able to solve it within 180 seconds, and as hard otherwise. We use this criteria and check how often easy instances turn out to be difficult for a approach, i.e., remain unsolved or require more than 180 seconds. Table~\ref{tab::bilevel::easyhard} summarizes these results over all ten test sets. 
\begin{table}[htbp]
\centering 
\footnotesize
\caption{
Aggregated solver performance with respect to benchmark difficulty labels. Reported are the numbers of hard instances solved within 180 seconds, easy instances solved within 180 seconds, and easy instances that remain unsolved or require more than 180 seconds.  Bold numbers indicate best values.
\label{tab::bilevel::easyhard}}
\begin{tabular}{rrrrr}
\toprule
&\Mibs&\Yasol (FA)&\Yasol (LA)&\Equips (LA) \\\midrule	
hard solved $\leq 180$s & 9 & 44 & 118  & \textbf{165}\\
easy solved $\leq 180$s & \textbf{181} & 131 & 158 & 147\\
easy not solved $\leq 180$s & \textbf{0} & 50 & 23 & 34
\\\bottomrule
\end{tabular}
\end{table}
\Mibs performs best on instances labeled as easy, which is not surprising since \Mibs was used in the process of classifying instances as easy. On LA-instances, both \Yasol and \Equips can solve several hard instances within three minutes. The main strength of \Equips has been the \texttt{bcpins} and \texttt{plusbcpins} testsets. In particular, each of the 80 \texttt{bcpins} instances (of which 70 are marked as hard) was solved within 60 seconds. Even though \Equips outperforms \Yasol (LA) on hard instances, it cannot solve as many easy instances in less than three minutes. Finally, Figure~\ref{fig::bilevel::cactus} shows the percentage of  instances solved with respect to the elapsed time. While \Equips is able to solve the most number of instances within the time limit, \Yasol (LA) can solve more than 14\% of the instances in under one second.    
\begin{figure}[htbp]
\centering
\includegraphics[width=.7\textwidth]{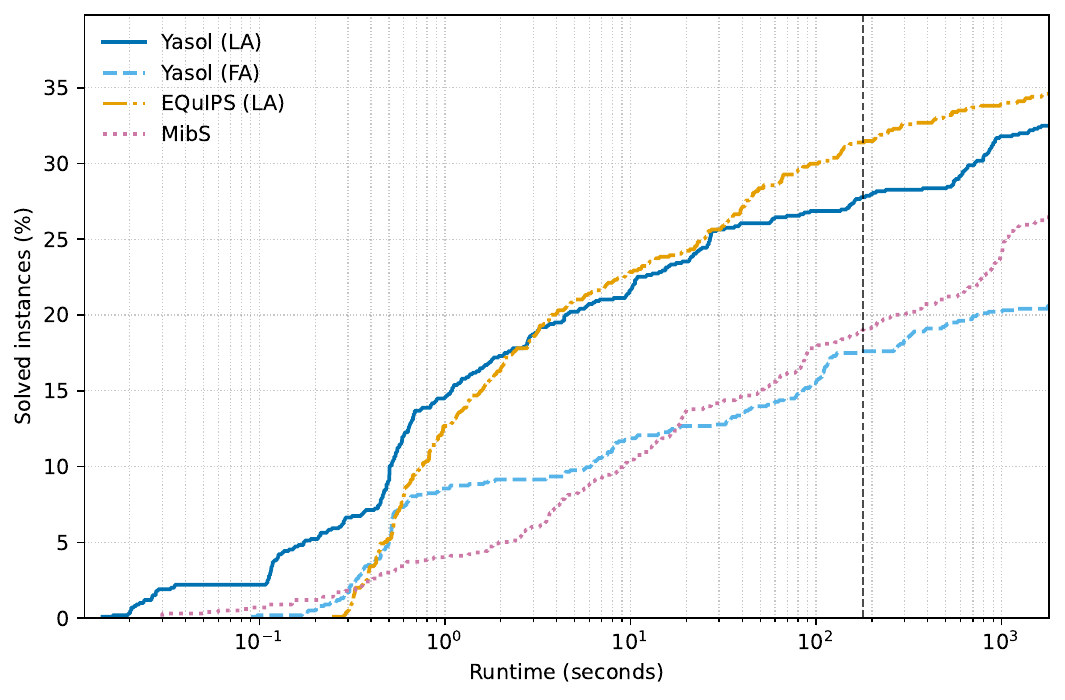}
\caption{Cactus plot showing the percentage of instances for which a solution was found within the given runtime over all test sets. The runtime axis is logarithmically scaled. The vertical dashed line marks the time limit for an instance being classified as easy according to BOBILib. \label{fig::bilevel::cactus}
}
\end{figure}

\subsection{Multistage Robust Cutting-Stock Problem}
\label{sec:ms-binpacking-benchmark}

As our final benchmark class, we consider a multistage robust cutting-stock problem with uncertain order sequences. In this setting, $B \in \mathbb{N}$ bars are available. Over time, orders for items arrive one after another and these items have to be cut from the bars. Each item belongs to one of $I \in \mathbb{N}$ item types, specified by its length. The total number $A_i$ of items of each type $i$ is known in advance, and hence the number of arrival stages is equal to $T=\sum_{i\in [I]} A_i$. At every time step, exactly one item order is revealed and must be assigned immediately to one of the bars. The decision cannot be postponed and cannot use information about future arrivals. This nonanticipativity requirement makes the problem a true multistage problem. The task is to choose assignments that work well not only for the current item, but also for all possible continuations of the order sequence: the goal is to minimize the worst-case cost over all admissible arrival orders. The objective combines three cost components: usage costs for bars, setup costs for switching to a previously unused bar, and waste costs for unusable leftover bars that become too short. We provide a detailed model description in Appendix~\ref{app:true-qip}.  For the experiments we created random instances. For each $T\in\{5,\ldots,12\}$ and $I\in\{2,\ldots,5\}$ we created 10 instances, resulting in overll $320$ instances. Details on the instance generation can be found in Appendix~\ref{app:instance-generation} and source code is available online\footnote{\url{https://github.com/MichaelHartisch/MultistageRobustCuttingStockProblem}}.

\paragraph{Multistage solvers vs. full expansion.}
In a first experiment, we compare the performance of \Yasol and \Equips on the compact QLP file formation with CPLEX 22.1.1 on the reformulation that fully expands all scenarios. This is often referred to as deterministic equivalent program (DEP) or full expansion. In Table~\ref{tab:qip-dep} we report instance file sizes, as well as the number of solved instances. While the \eqref{problem} formulation in the QLP file format only grows linear in size, the DEP in the LP file format grows exponentially. It can be seen that Yasol's runtime only moderately increases for larger instances. \Equips, which iteratively builds up a partial expansion, potentially ending up with the full expansion, performs worst.

\begin{table}[htbp]
\footnotesize
    \centering
    \caption{Instance file sizes and number of solved instances as well as the overall runtime (rounded to next integer) of \Yasol, \Equips, and \Cplex on multistage robust cutting-stock instances. \label{tab:qip-dep}} 

    \begin{tabular}{rrr|rr|rr|rr}
        \toprule
    &    \multicolumn{2}{c|}{max file sizes} & \multicolumn{2}{c|}{\Yasol} & \multicolumn{2}{c|}{\Equips} & \multicolumn{2}{c}{\Cplex}\\
\(T\) &  QLP  & DEP & \# solved &time[s]& \# solved &time[s]& \# solved &time[s]  \\
        \midrule
5 & 8 & 116 & 40 & 3 & 40 & 1715& 40 & 1 \\
6 & 16 & 1040 & 40 & 9 & 27 & 26147& 40 & 126 \\
7 & 16 & 3763 & 40 & 22 & 18 & 44848& 35 & 11359\\
8 & 24 & 20660 & 40 & 115 & 10 &59775 & 16 & 45270\\
9 & 28 & 64160 & 40 & 348 &4 & 65902& 15 &47890\\
10 & 36 & 412556 & 40 & 2415 & 11 & 57451 & 18 & 41141 \\
11 & 40 & 1560888 & 40 & 7999 & 0 & 72000 & 4 & 65564 \\
12 & 52 & 7711608 & 23 & 38684 & 2 & 68524 & 4 & 64871 \\
        \bottomrule
    \end{tabular}
\end{table}

\paragraph{Value of the multistage structure.}
In a second experiment, we evaluate the cost of ignoring the multistage nature of the problem. We consider two relaxed versions of the multistage problem: A single-stage approach, computing an assignment of items to bars completely ignoring the uncertain order sequences, and a two-stage approach, that aims at finding such a fixed assignment that, when implemented, has minimal cost with respect to the worst-case order sequence. In Appendix~\ref{app:single-two-stage} we provide detailed descriptions of these models. Both approaches ignore that in the actual multistage setting, one does not have to stick to a fixed plan of assignments of items to bars, but can make different assignments depending on the given order sequence. These plans are then evaluated in the true multistage setting (also see Appendix~\ref{app:single-two-stage}). The single-stage approach yields an IP, which is solved via CPLEX 22.1.1. The two-stage approach is solved by using \Yasol; we also used \Equips and tried solving the respective full expansion with CPLEX, but \Yasol performed best on these instances.  In Figure~\ref{fig::bin::NoMulti} we report the relative deviation from the multistage optimum, i.e., $\Delta =
\frac{z - z^\star_{\mathrm{MS}}}{z^\star_{\mathrm{MS}}}\cdot 100$,
where $z^\star_{\mathrm{MS}}$ is the optimal value of the true multistage
problem and $z$ is the value obtained by the respective method. Hence, a value of zero means that the method attains the multistage optimum, while larger values indicate worse worst-case performance.
\begin{figure}[htbp]
    \centering
\includegraphics[width=.5\textwidth]{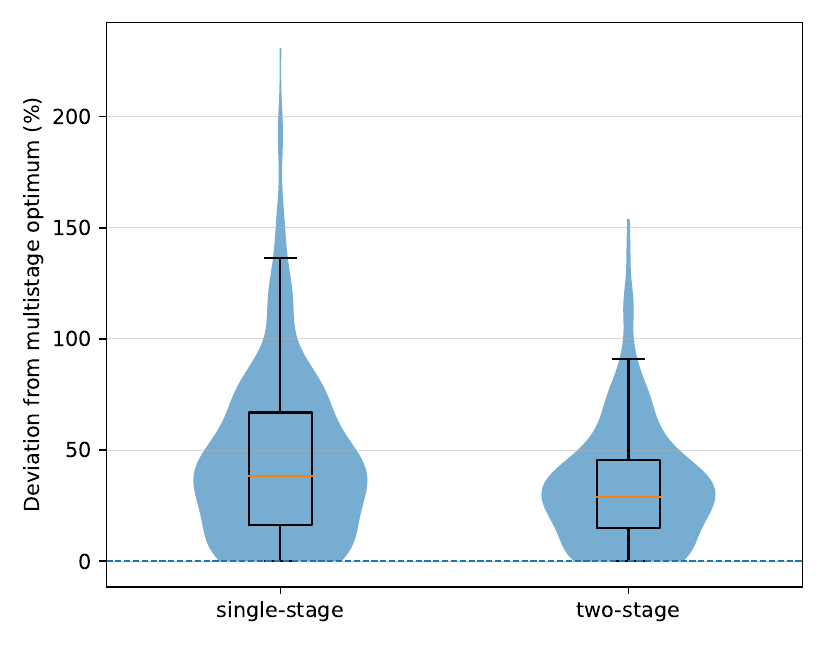}
    \caption{Distribution of the relative deviation from the multistage optimum for the single-stage and two-stage approaches after evaluation in the true
multistage setting. The embedded box plots show the median and spread.\label{fig::bin::NoMulti}}
\end{figure}
As expected, the two-stage approach performs better than the single-stage approach. Both yield plans that  are multistage-optimal on $14.8\%$ of the
instances, but the two-stage approach yields smaller deviations overall. All single-stage plans were computed in not more than $0.02$ seconds, while computing the solutions of the two-stage approach with \Yasol required overall more time than solving the actual multi-stage problem with \Yasol. Hence, while the single-stage approach has a massive computational advantage, its performance in the true multistage setting is often very weak. The two-stage approach partially accounts for the  sequential structure, while still leaving a substantial gap to the true multistage optimum and not even resulting in an immediate runtime advantage in many cases. Here, however, developing sophisticated problem-specific approaches likely can yield a better performance.

\paragraph{Comparison with decision rules.} 

Finally, we compare our exact approach with the adaptive binary decision rules of Bertsimas and Georghiou~\cite{bertsimas2015design}. In this framework, each later-stage binary decision is represented by a parameterized threshold rule defined as the difference of two piecewise-linear functions. More precisely, for a given number $P$ of pieces, the decision rule takes the form
\[
a^t(\xivec) =
\begin{cases}
1,
&
\textnormal{if }
\max\limits_{p=1,\ldots,P}\{ \bar y_p^\top \xivec + \bar y_{p,0}\}
-
\max\limits_{p=1,\ldots,P}\{ \hat y_p^\top \xivec + \hat y_{p,0}\}
\le 0,\\
0,
&
\textnormal{otherwise,}
\end{cases}
\]
where $\xivec = (\xivec^t,\ldots,\xivec^{t-h})$ denotes the part of the observed history used by the rule. In particular, each $\xivec^\tau$ is the observation vector revealed at stage $\tau$, and
$h\in\mathbb{N}$ determines the length of the information window. In our specific case $\xivec^\tau$ is a binary vector with a single $1$ entry, indicating which item type was ordered in stage $\tau$. We note that for $P=1$, the rule reduces to a linear threshold rule.

We implemented the iterative scenario-generation scheme proposed in \cite{bertsimas2015design} using CPLEX 22.1.1 as the underlying MIP solver. We consider $P\in\{1,2\}$ and $h \in \{1,2,3,4\}$ and denote the corresponding variants by $\mathrm{DR}(P,h)$. In the experiments, we only consider instances for which the exact multistage approach proves an optimal, feasible solution. Figure~\ref{fig::bin::DRtimes} reports for each method the fraction of instances solved within the given runtime. For \Yasol, an instance is counted as soon as the optimal solution is found. For the decision-rule variants, an instance is counted once a feasible decision rule is certified. Thus, timeouts as well as certified nonexistence of a decision rule are not counted as successful policy generation in this plot.
\begin{figure}[htbp]
\centering
\includegraphics[width=.7\textwidth]{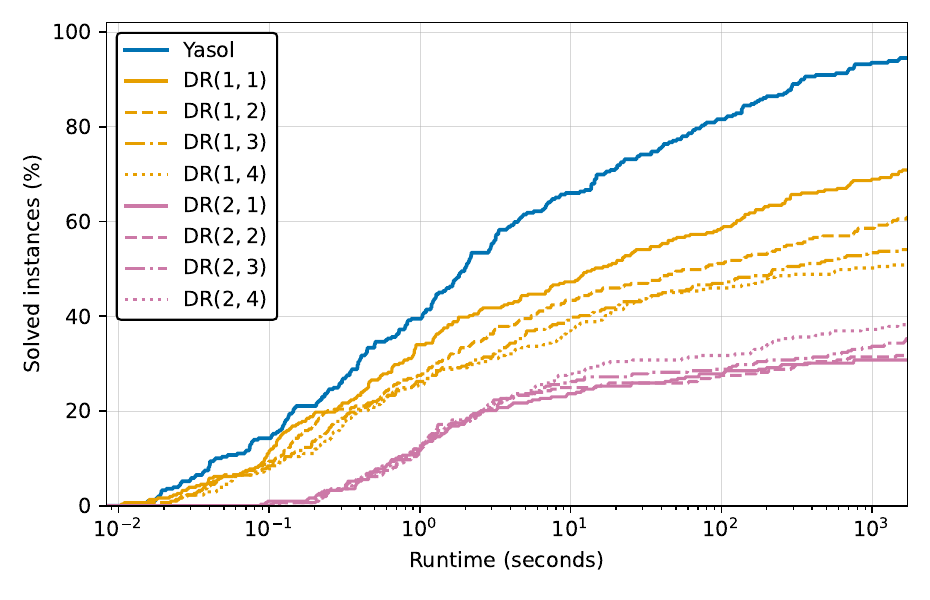}
\caption{Cactus plot showing the percentage of instances for which a solution was found within the given runtime. The runtime axis is logarithmically scaled. For the decision-rule variants, only runs with a found decision rule are counted; timeouts and infeasibility are excluded. 
\label{fig::bin::DRtimes}}
\end{figure}
It can be seen that \Yasol solves significantly more instances to optimality within the time limit. In contrast, even $\mathrm{DR}(1,1)$ requires substantially more runtime and finds a decision rule for only about $70\%$ of the instances. For $P=1$ and increasing $h$, this ratio further decreases. For $P=2$, in less than $40\%$ of the instances a decision rule is obtained, while, interestingly, for increasing $h$ this ratio tends to increase. Overall, for the considered test set, the decision-rule approach does not provide a computational advantage in terms of runtime. Nevertheless, decision rules offer a different type of output: they provide a complete policy that maps observed orders to decisions. In contrast, \Yasol provides the optimal first-stage decision together with the actions corresponding to the generated worst-case scenario.

Furthermore, we analyze the outcomes of the decision-rule approach with respect to their quality. If a decision rule is found, we compare its objective value with the multistage optimum. In Figure~\ref{fig::bin::DRbuckets} we report the different outcomes for $\mathrm{DR}(1,h)$.
\begin{figure}[htbp]
\centering
\includegraphics[width=\textwidth]{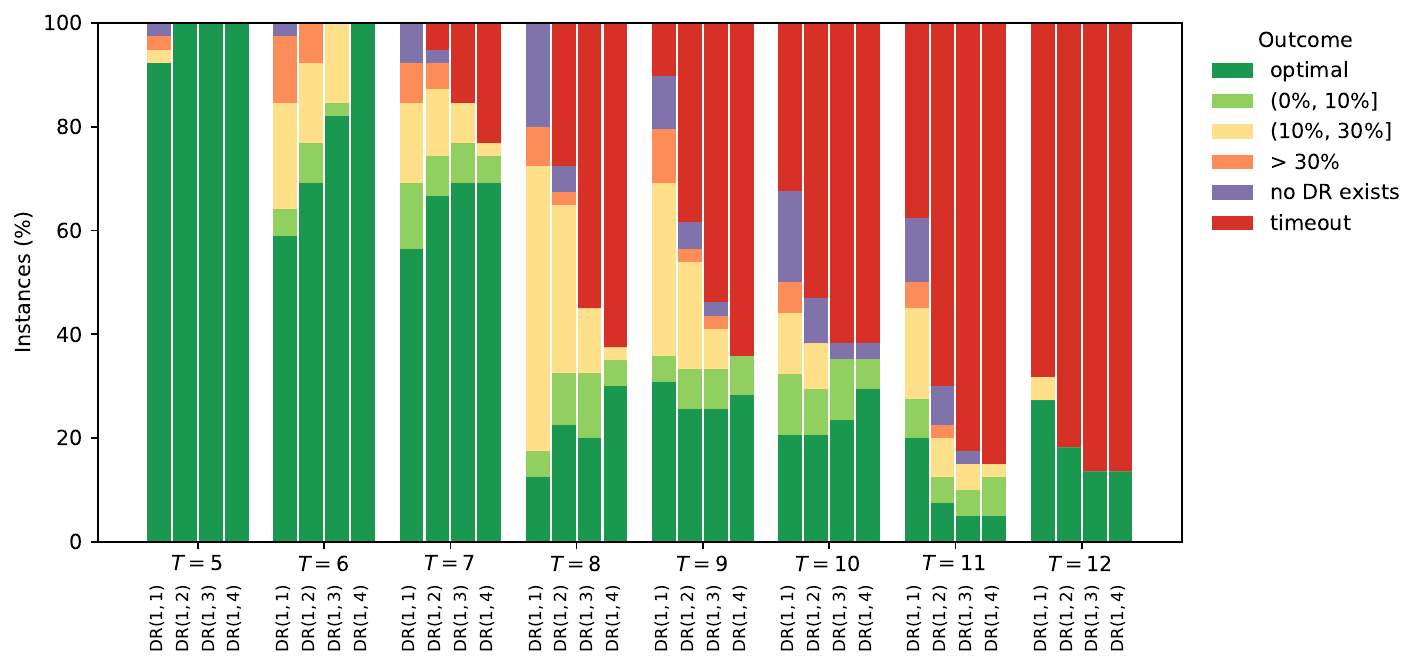}
\caption{
Outcome distribution of the decision-rule variants $\mathrm{DR}(1,h)$. Found decision rules are classified as multistage optimal or grouped by their deviation from the multistage optimum. The remaining categories report instances for which no decision-rule solution is obtained, either because the time limit is reached or because no decision rule of the considered type exists.
\label{fig::bin::DRbuckets}}
\end{figure}
As expected, increasing $h$ generally improves the quality of the obtained decision rules. At the same time, larger values of $h$ lead to more timeouts as the optimization problem of finding a decision rule gets more complex. For small numbers of decision stages $T$, the decision-rule approach is often able to find a rule that performs optimally. However, for large $T$, we often were not able to find a decision-rule within the time limit. In summary, the decision-rule approach can provide useful policies. However, on this test set, decision rules do not provide a computational advantage over our approach. Recall that \Yasol solved all instances up to $T=11$ to optimality while, e.g., for $\mathrm{DR}(1,1)$ in only around $50\%$ of the instances a decision rule is obtained. Hence, \Yasol is able to certify the true multistage optimum more quickly, although it does not provide a complete policy. 

\section{Conclusions\label{sec::outlook}}

We have presented a search-based solution approach for multistage robust discrete linear optimization problems that adapts and integrates techniques from game-tree search, integer linear programming, and quantified constraint satisfaction. We introduced our solver Yasol, which handles problems with final-stage mixed-integer recourse and discrete decision-dependent uncertainty. By providing a model-and-run alternative to problem-specific reformulations and dualization techniques, Yasol avoids potentially tedious reformulations that may otherwise need to be revised whenever the underlying model changes.

Our computational study demonstrates the potential of Yasol as a reference solver for multistage robust discrete linear optimization. The results show competitive performance against a baseline approach for a protection-interdiction problem and specialized algorithms on bilevel benchmark instances, outperforming them in some settings. Furthermore, experiments on a multistage robust cutting-stock problem demonstrate that neglecting the multistage structure can lead to substantial losses in solution quality. For the same problem, we show that the considered decision-rule approximations, while providing explicit decision policies, offer no computational advantage over our exact approach, which instead certifies the true multistage optimum.

Several challenges remain for future research, offering opportunities to further improve the performance of the proposed solution approach. In particular, relaxation and branching techniques should be developed that better exploit the structure of decision-dependent uncertainty. Moreover, the potential of learning adversarial constraints and propagating adversarial variables remains to be exploited more effectively. Another promising direction is the integration of our search-based approach with complementary expansion-based methods, aiming to combine the strengths of both paradigms.

 \section*{Code availability}
Scripts to reproduce the conducted experiments are available in the github repository \url{https://github.com/MichaelHartisch/yasol-experiments}.

\bibliographystyle{plain}      
\bibliography{references}   

@article{tange2011gnu,
  title={Gnu parallel-the command-line power tool},
  author={Tange, Ole},
  journal={Usenix Mag},
  volume={36},
  number={1},
  pages={42},
  year={2011}
}

@article{tan2025adjustable,
  title={Adjustable robust optimization with decision-dependent uncertainty for power system problems: A review},
  author={Tan, Tao and Yang, Meng and Xie, Rui and Cao, Yuji and Chen, Yue},
  journal={Energy Conversion and Economics},
  volume={6},
  number={3},
  pages={155--169},
  year={2025},
  publisher={Wiley Online Library}
}

@article{haitao2025robust,
  title={Robust facility location considering protection and backup under facility disruptions with decision-dependent uncertainty},
  author={Haitao, HU and Jiafu, TANG},
  journal={European Journal of Operational Research},
  volume={325},
  number={3},
  pages={474--486},
  year={2025},
  publisher={Elsevier}
}

@misc{cplex,
  author       = {{IBM}},
  title        = {{IBM ILOG CPLEX Optimizer}},
  howpublished = {Optimization software},
  url          = {https://www.ibm.com/products/ilog-cplex-optimization-studio}
}

@article{wrathall1976complete,
  title={Complete sets and the polynomial-time hierarchy},
  author={Wrathall, Celia},
  journal={Theoretical Computer Science},
  volume={3},
  number={1},
  pages={23--33},
  year={1976},
  publisher={Elsevier}
}

@inproceedings{chistikov2017complexity,
  title={On the complexity of quantified integer programming},
  author={Chistikov, Dmitry and Haase, Christoph},
  booktitle={44th International Colloquium on Automata, Languages, and Programming (ICALP 2017)},
  volume={80},
  pages={94},
  year={2017},
  organization={Schloss Dagstuhl--Leibniz-Zentrum fuer Informatik}
}

@article{gao2020solving,
  title={Solving quantified constraint satisfaction problems with value selection rules},
  author={Gao, Jian and Wang, Jinyan and Wu, Kuixian and Chen, Rong},
  journal={Frontiers of Computer Science},
  volume={14},
  number={5},
  pages={145317},
  year={2020},
  publisher={Springer}
}

@article{bertsimas2018binary,
  title={Binary decision rules for multistage adaptive mixed-integer optimization},
  author={Bertsimas, Dimitris and Georghiou, Angelos},
  journal={Mathematical Programming},
  volume={167},
  number={2},
  pages={395--433},
  year={2018},
  publisher={Springer}
}

@inproceedings{mamoulis2004algorithms,
  title={Algorithms for quantified constraint satisfaction problems},
  author={Mamoulis, Nikos and Stergiou, Kostas},
  booktitle={International Conference on Principles and Practice of Constraint Programming},
  pages={752--756},
  year={2004},
  organization={Springer}
}

@inproceedings{benedetti2008quantified,
  title={Quantified constraint optimization},
  author={Benedetti, Marco and Lallouet, Arnaud and Vautard, J{\'e}r{\'e}mie},
  booktitle={International Conference on Principles and Practice of Constraint Programming},
  pages={463--477},
  year={2008},
  organization={Springer}
}

@article{lougee2003common,
  title={The {C}ommon {O}ptimization {IN}terface for {O}perations {R}esearch: Promoting open-source software in the operations research community},
  author={R. Lougee-Heimer},
  journal={IBM Journal of Research and Development},
  volume={47},
  number={1},
  pages={57--66},
  year={2003},
  publisher={IBM}
}

@InProceedings{hartisch_et_al:LIPIcs.CP.2025.12,
  author =	{Hartisch, Michael and Chew, Leroy},
  title =	{{An Expansion-Based Approach for Quantified Integer Programming}},
  booktitle =	{31st International Conference on Principles and Practice of Constraint Programming (CP 2025)},
  pages =	{12:1--12:26},
  series =	{Leibniz International Proceedings in Informatics (LIPIcs)},
  ISBN =	{978-3-95977-380-5},
  ISSN =	{1868-8969},
  year =	{2025},
  volume =	{340},
  editor =	{de la Banda, Maria Garcia},
  publisher =	{Schloss Dagstuhl -- Leibniz-Zentrum f{\"u}r Informatik},
  address =	{Dagstuhl, Germany},
  URL =		{https://drops.dagstuhl.de/entities/document/10.4230/LIPIcs.CP.2025.12},
  URN =		{urn:nbn:de:0030-drops-238736},
  doi =		{10.4230/LIPIcs.CP.2025.12}
}

@article{huangfu2018parallelizing,
  title={Parallelizing the dual revised simplex method},
  author={Huangfu, Qi and Hall, JA Julian},
  journal={Mathematical Programming Computation},
  volume={10},
  number={1},
  pages={119--142},
  year={2018},
  publisher={Springer}
}

@article{gerber1995parametric,
  title={Parametric dispatching of hard real-time tasks},
  author={Gerber, R. and Pugh, W. and Saksena, M.},
  journal={IEEE transactions on computers},
  volume={44},
  number={3},
  pages={471--479},
  year={1995},
  publisher={IEEE}
}

@article{Euro15,
  title={Solving multistage quantified linear optimization problems with the alpha--beta nested benders decomposition},
  author={Lorenz, U. and Wolf, J.},
  journal={EURO Journal on Computational Optimization},
  volume={3},
  number={4},
  pages={349--370},
  year={2015},
  publisher={Springer}
}

@article{chen2020robust,
  title={Robust stochastic optimization made easy with RSOME},
  author={Chen, Z. and Sim, M. and Xiong, P.},
  journal={Management Science},
  volume={66},
  number={8},
  pages={3329--3339},
  year={2020},
  publisher={INFORMS}
}

@incollection{denegre2009branch,
  title={A branch-and-cut algorithm for integer bilevel linear programs},
  author={DeNegre, Scott T and Ralphs, Ted K},
  booktitle={Operations research and cyber-infrastructure},
  pages={65--78},
  year={2009},
  publisher={Springer}
}

@article{tahernejad2020branch,
  title={A branch-and-cut algorithm for mixed integer bilevel linear optimization problems and its implementation},
  author={Tahernejad, S. and Ralphs, T.K. and DeNegre, S.T.},
  journal={Mathematical Programming Computation},
  volume={12},
  number={4},
  pages={529--568},
  year={2020},
  publisher={Springer}
}

@article{fischetti2017new,
  title={A new general-purpose algorithm for mixed-integer bilevel linear programs},
  author={Fischetti, M. and Ljubi{\'c}, I. and Monaci, M. and Sinnl, M.},
  journal={Operations Research},
  volume={65},
  number={6},
  pages={1615--1637},
  year={2017},
  publisher={INFORMS}
}

@article{bolusani2022framework,
  title={A framework for generalized Benders’ decomposition and its application to multilevel optimization},
  author={Bolusani, S. and Ralphs, T.K.},
  journal={Mathematical Programming},
  volume={196},
  number={1},
  pages={389--426},
  year={2022},
  publisher={Springer}
}

@article{OJMO_2025__6__A2_0,
     author = {Marc Goerigk and Jannis Kurtz and Martin Schmidt and Johannes Th\"urauf},
     title = {Connections between {Robust} and {Bilevel} {Optimization}},
     journal = {Open Journal of Mathematical Optimization},
     eid = {2},
     pages = {1--17},
     year = {2025},
     publisher = {Universit\'e de Montpellier},
     volume = {6},
     doi = {10.5802/ojmo.38},
     language = {en},
     url = {https://ojmo.centre-mersenne.org/articles/10.5802/ojmo.38/}
}

@article{goerigk2024introduction,
  title={An introduction to robust combinatorial optimization},
  author={Goerigk, Marc and Hartisch, Michael},
  journal={International Series in Operations Research and Management Science},
  year={2024},
  publisher={Springer}
}

@article{sato2021gradient,
  title={A gradient method for multilevel optimization},
  author={Sato, R. and Tanaka, M. and Takeda, A.},
  journal={Advances in Neural Information Processing Systems},
  volume={34},
  pages={7522--7533},
  year={2021}
}

@article{vayanos2020roc++,
  title={Roc++: Robust optimization in c++},
  author={Vayanos, P. and Jin, Q. and Elissaios, G.},
  journal={INFORMS Journal on Computing},
  volume={34},
  number={6},
  pages={2873--2888},
  year={2022},
  publisher={INFORMS}
}

@inproceedings{Killer,
  author    = {S.G. Akl and
               M.M. Newborn},
  title     = {The principal continuation and the killer heuristic},
  booktitle = {Proceedings of the 1977 annual conference, {ACM} '77, Seattle, Washington,
               USA},
  pages     = {466--473},
  year      = {1977},
  bibsource = {dblp computer science bibliography, http://dblp.org}
}

@inproceedings{moskewicz2001chaff,
  title={Chaff: Engineering an efficient {SAT} solver},
  author={Moskewicz, M.W. and Madigan, C.F. and Zhao, Y. and Zhang, L. and Malik, S.},
  booktitle={Proceedings of the 38th annual Design Automation Conference},
  pages={530--535},
  year={2001}
}

@phdthesis{dunning2016advances,
  title={Advances in robust and adaptive optimization: algorithms, software, and insights},
  author={Dunning, I.},
  year={2016},
  school={Massachusetts Institute of Technology}
}

@article{thiele2009robust,
  title={Robust linear optimization with recourse},
  author={Thiele, A. and Terry, T. and Epelman, M.},
  journal={Rapport Technique},
  pages={4--37},
  year={2009}
}

@article{yanikouglu2019survey,
  title={A survey of adjustable robust optimization},
  author={Yan{\i}ko{\u{g}}lu, {\.I}hsan and Gorissen, Bram L and den Hertog, Dick},
  journal={European Journal of Operational Research},
  volume={277},
  number={3},
  pages={799--813},
  year={2019},
  publisher={Elsevier}
}

@article{postek2016multistage,
  title={Multistage adjustable robust mixed-integer optimization via iterative splitting of the uncertainty set},
  author={Postek, K. and den Hertog, D.},
  journal={INFORMS Journal on Computing},
  volume={28},
  number={3},
  pages={553--574},
  year={2016},
  publisher={INFORMS}
}

@article{shapiro2011dynamic,
  title={A dynamic programming approach to adjustable robust optimization},
  author={Shapiro, A.},
  journal={Operations Research Letters},
  volume={39},
  number={2},
  pages={83--87},
  year={2011},
  publisher={Elsevier}
}

@inproceedings{subramani2003analysis,
  title={An analysis of quantified linear programs},
  author={Subramani, K.},
  booktitle={International Conference on Discrete Mathematics and Theoretical Computer Science},
  pages={265--277},
  year={2003},
  organization={Springer}
}

@inproceedings{subramani2004analyzing,
  title={Analyzing selected quantified integer programs},
  author={Subramani, K.},
  booktitle={International Joint Conference on Automated Reasoning},
  pages={342--356},
  year={2004},
  organization={Springer}
}

@article{ben2004adjustable,
  title={Adjustable robust solutions of uncertain linear programs},
  author={Ben-Tal, A. and Goryashko, A. and Guslitzer, E. and Nemirovski, A.},
  journal={Mathematical Programming},
  volume={99},
  number={2},
  pages={351--376},
  year={2004},
  publisher={Springer}
}

@article{gorissen2015practical,
  title={A practical guide to robust optimization},
  author={Gorissen, B.L. and Yan{\i}ko{\u{g}}lu, {\.I}. and den Hertog, D.},
  journal={Omega},
  volume={53},
  pages={124--137},
  year={2015},
  publisher={Elsevier}
}

@article{ben2002robust,
  title={Robust optimization - methodology and applications},
  author={Ben-Tal, A. and Nemirovski, A.},
  journal={Mathematical Programming},
  volume={92},
  number={3},
  pages={453--480},
  year={2002},
  publisher={Springer}
}

@article{gabrel2014recent,
  title={Recent advances in robust optimization: An overview},
  author={Gabrel, V. and Murat, C. and Thiele, A.},
  journal={European Journal of Operational Research},
  volume={235},
  number={3},
  pages={471--483},
  year={2014},
  publisher={Elsevier}
}

@inproceedings{janota2015solving,
  title={Solving {QBF} by clause selection},
  author={Janota, M. and Marques-Silva, J.},
  booktitle={Proceedings of the 24th International Joint Conference on Artificial Intelligence},
  pages={325--331},
  year={2015}
}

@inproceedings{lonsing2017depqbf,
  title={{DepQBF} 6.0: A search-based {QBF} solver beyond traditional {QCDCL}},
  author={Lonsing, F. and Egly, U.},
  booktitle={International Conference on Automated Deduction},
  pages={371--384},
  year={2017},
  organization={Springer International Publishing, Cham}
}

@article{heule2015clause,
  title={Clause elimination for {SAT} and {QSAT}},
  author={Heule, M. and J{\"a}rvisalo, M. and Lonsing, F. and Seidl, M. and Biere, A.},
  journal={Journal of Artificial Intelligence Research},
  volume={53},
  pages={127--168},
  year={2015}
}

@inproceedings{lonsing2015enhancing,
  title={Enhancing search-based {QBF} solving by dynamic blocked clause elimination},
  author={Lonsing, F. and Bacchus, F. and Biere, A. and Egly, U. and Seidl, M.},
  booktitle={Logic for Programming, Artificial Intelligence, and Reasoning},
  pages={418--433},
  year={2015},
  organization={Springer Heidelberg Berlin}
}

@inproceedings{janota2018towards,
  title={Towards generalization in {QBF} solving via machine learning},
  author={Janota, M.},
  booktitle={Proceedings of the AAAI Conference on Artificial Intelligence},
  year={2018}
}

@inproceedings{wimmer2017hqspre,
  title={HQSpre--an effective preprocessor for {QBF} and {DQBF}},
  author={Wimmer, R. and Reimer, S. and Marin, P. and Becker, B.},
  booktitle={International Conference on Tools and Algorithms for the Construction and Analysis of Systems},
  pages={373--390},
  year={2017},
  organization={Springer Berlin Heidelberg}
}

@article{leitner2021exact,
  title={An exact method for binary fortification games},
  author={Leitner, Markus and Ljubi{\'c}, Ivana and Monaci, Michele and Sinnl, Markus and Tan{\i}nm{\i}{\c{s}}, K{\"u}bra},
  journal={European Journal of Operational Research},
  volume={307},
  number={3},
  pages={1026--1039},
  year={2023},
  publisher={Elsevier}
}

@article{cadoli2002algorithm,
  title={An algorithm to evaluate quantified boolean formulae and its experimental evaluation},
  author={Cadoli, M. and Schaerf, M. and Giovanardi, A. and Giovanardi, M.},
  journal={Journal of Automated Reasoning},
  volume={28},
  number={2},
  pages={101--142},
  year={2002},
  publisher={Springer}
}

@article{tentrup2019caqe,
  title={CAQE and QuAbS: Abstraction Based {QBF} Solvers},
  author={Tentrup, L.},
  journal={Journal on Satisfiability, Boolean Modeling and Computation},
  volume={11},
  number={1},
  pages={155--210},
  year={2019},
  publisher={IOS Press}
}

@incollection{dempe2020bilevel,
  title={Bilevel optimization},
  author={Dempe, Stephan and Zemkoho, Alain},
  booktitle={Springer optimization and its applications},
  volume={161},
  year={2020},
  publisher={Springer}
}

@book{migdalas2013multilevel,
  title={Multilevel optimization: algorithms and applications},
  author={Migdalas, Athanasios and Pardalos, Panos M and V{\"a}rbrand, Peter},
  volume={20},
  year={2013},
  publisher={Springer Science \& Business Media}
}

@inproceedings{lonsing2019qratpre,
  title={{QRATPre+}: Effective {QBF} preprocessing via strong redundancy properties},
  author={Lonsing, F. and Egly, U.},
  booktitle={International Conference on Theory and Applications of Satisfiability Testing},
  pages={203--210},
  year={2019},
  organization={Springer International Publishing, Cham}
}

@article{wiebe2021romodel,
  title={{ROmodel}: Modeling robust optimization problems in Pyomo},
  author={Wiebe, J. and Misener, R.},
  journal={Optimization and Engineering},
  pages={1--22},
  year={2021},
  publisher={Springer}
}

@article{goh2011robust,
  title={Robust optimization made easy with {ROME}},
  author={Goh, J. and Sim, M.},
  journal={Operations Research},
  volume={59},
  number={4},
  pages={973--985},
  year={2011},
  publisher={INFORMS}
}

@article{bertsimas2004price,
  title={The price of robustness},
  author={Bertsimas, D. and Sim, M.},
  journal={Operations research},
  volume={52},
  number={1},
  pages={35--53},
  year={2004},
  publisher={Informs}
}

@article{witzig2021computational,
  title={Computational aspects of infeasibility analysis in mixed integer programming},
  author={Witzig, J. and Berthold, T. and Heinz, S.},
  journal={Mathematical Programming Computation},
  volume={13},
  number={4},
  pages={753--785},
  year={2021},
  publisher={Springer}
}

@article{benders1962partitioning,
  title={Partitioning procedures for solving mixed-variables programming problems ‘},
  author={Benders, J.F.},
  journal={Numerische Mathematik},
  volume={4},
  number={1},
  pages={238--252},
  year={1962}
}

@article{gent2008solving,
  title={Solving quantified constraint satisfaction problems},
  author={Gent, I.P. and Nightingale, P. and Rowley, A. and Stergiou, K.},
  journal={Artificial Intelligence},
  volume={172},
  number={6-7},
  pages={738--771},
  year={2008},
  publisher={Elsevier}
}

@article{Qube,
title = "Backjumping for Quantified Boolean Logic satisfiability",
journal = "Artificial Intelligence",
volume = "145",
number = "1",
pages = "99--120",
year = "2003",
issn = "0004-3702",
doi = "https://doi.org/10.1016/S0004-3702(02)00373-9",
url = "http://www.sciencedirect.com/science/article/pii/S0004370202003739",
author = "E. Giunchiglia and M. Narizzano and A. Tacchella"
}

@article{bertsimas2015design,
  title={Design of near optimal decision rules in multistage adaptive mixed-integer optimization},
  author={Bertsimas, D. and Georghiou, A.},
  journal={Operations Research},
  volume={63},
  number={3},
  pages={610--627},
  year={2015},
  publisher={INFORMS}
}

@article{dolan2002benchmarking,
  title={Benchmarking optimization software with performance profiles},
  author={Dolan, E.D. and Mor{\'e}, J.J.},
  journal={Mathematical programming},
  volume={91},
  number={2},
  pages={201--213},
  year={2002},
  publisher={Springer}
}

@article{bertsimas2016multistage,
  title={Multistage robust mixed-integer optimization with adaptive partitions},
  author={Bertsimas, D. and Dunning, I.},
  journal={Operations Research},
  volume={64},
  number={4},
  pages={980--998},
  year={2016},
  publisher={INFORMS}
}

@phdthesis{AchterbergDiss,
  author      = {Tobias Achterberg},
  title       = {Constraint Integer Programming},
  school      = {TU Berlin},
  year        = {2009},
}

@article{achterberg2009scip,
  title={{SCIP}: solving constraint integer programs},
  author={Achterberg, T.},
  journal={Mathematical Programming Computation},
  volume={1},
  number={1},
  pages={1--41},
  year={2009},
  publisher={Springer}
}

@article{achterberg2005branching,
  title={Branching rules revisited},
  author={Achterberg, T. and Koch, T. and Martin, A.},
  journal={Operations Research Letters},
  volume={33},
  number={1},
  pages={42--54},
  year={2005},
  publisher={Elsevier}
}

@inproceedings{hartisch2021adaptive,
  title={Adaptive Relaxations for Multistage Robust Optimization},
  author={Hartisch, M.},
  booktitle={Pacific Rim International Conference on Artificial Intelligence},
  pages={485--499},
  year={2021},
  organization={Springer}
}

@inproceedings{zhang2002conflict,
  title={Conflict driven learning in a quantified Boolean satisfiability solver},
  author={Zhang, L. and Malik, S.},
  booktitle={Proceedings of the 2002 IEEE/ACM international conference on Computer-aided design},
  pages={442--449},
  year={2002}
}

@inproceedings{lonsing2013efficient,
  title={Efficient clause learning for quantified boolean formulas via QBF pseudo unit propagation},
  author={Lonsing, F. and Egly, U. and van Gelder, A.},
  booktitle={International Conference on Theory and Applications of Satisfiability Testing},
  pages={100--115},
  year={2013},
  organization={Springer}
}

@article{achterberg2007conflict,
  title={Conflict analysis in mixed integer programming},
  author={Achterberg, T.},
  journal={Discrete Optimization},
  volume={4},
  number={1},
  pages={4--20},
  year={2007},
  publisher={Elsevier}
}

@article{Thuerauf_et_al:2026,
        title    = {{BOBILib}: Bilevel Optimization (Benchmark) Instance Library},
        author   = {Johannes Thürauf and Thomas Kleinert and Ivana Ljubi\'c and Ted Ralphs and Martin Schmidt},
        journal  = {Mathematical Programming Computation},
        doi      = {10.1007/s12532-025-00294-y},
        year     = {2026},
        isbn     = {1867-2957},
      }

@article{lefebvre2025solving,
  title={Solving Decision-Dependent Robust Problems as Bilevel Optimization Problems},
  author={Lefebvre, Henri and Schmidt, Martin and Stevens, Simon and Th{\"u}rauf, Johannes},
  journal={arXiv preprint arXiv:2503.01559},
  year={2025}
}

@article{lefebvre2026branch,
  title={Branch-and-Cut for Mixed-Integer Linear Decision-Dependent Robust Optimization},
  author={Lefebvre, Henri and Schmidt, Martin and Stevens, Simon and Th{\"u}rauf, Johannes},
  journal={Optimization Online},
  url={https://optimization-online.org/?p=33886},
  year={2026}
}

@article{gomes1998boosting,
  title={Boosting combinatorial search through randomization},
  author={Gomes, Carla P and Selman, Bart and Kautz, Henry and others},
  journal={AAAI/IAAI},
  volume={98},
  number={1998},
  pages={431--437},
  year={1998}
}

@incollection{beyersdorff2021quantified,
  title={Quantified Boolean formulas},
  author={Beyersdorff, O. and Janota, M. and Lonsing, F. and Seidl, M.},
  booktitle={Handbook of Satisfiability},
  pages={1177--1221},
  year={2021},
  publisher={IOS Press}
}

@inproceedings{YasolACG17,
  title={Yasol: An open source solver for quantified mixed integer programs},
  author={Ederer, T. and Hartisch, M. and Lorenz, U. and Opfer, T. and Wolf, J.},
  booktitle={Advances in Computer Games},
  pages={224--233},
  year={2017},
  organization={Springer}
}

@article{goerigk2021multistage,
  title={Multistage robust discrete optimization via quantified integer programming},
  author={Goerigk, M. and Hartisch, M.},
  journal={Computers \& Operations Research},
  volume = {135},
  number = {105434},
  pages={1--13},
  year={2021},
  publisher={Elsevier}
}

@inproceedings{hartisch2019mastering,
  title={Mastering uncertainty: towards robust multistage optimization with decision dependent uncertainty},
  author={Hartisch, M. and Lorenz, U.},
  booktitle={Pacific Rim International Conference on Artificial Intelligence},
  pages={446--458},
  year={2019},
  organization={Springer}
}

@article{lappas2018robust,
  title={Robust optimization for decision-making under endogenous uncertainty},
  author={Lappas, N.H. and Gounaris, C.E.},
  journal={Computers \& Chemical Engineering},
  volume={111},
  pages={252--266},
  year={2018},
  publisher={Elsevier}
}

@article{nohadani2018optimization,
  title={Optimization under decision-dependent uncertainty},
  author={Nohadani, O. and Sharma, K.},
  journal={SIAM Journal on Optimization},
  volume={28},
  number={2},
  pages={1773--1795},
  year={2018},
  publisher={SIAM}
}

@InProceedings{hartisch2019novel,
author="Hartisch, M.
and Lorenz, U.",
editor="Cazenave, T.
and van den Herik, J.
and Saffidine, A.
and Wu, I-C.",
title="A Novel Application for Game Tree Search - Exploiting Pruning Mechanisms for Quantified Integer Programs",
booktitle="Advances in Computer Games",
year="2020",
publisher="Springer International Publishing",
address="Cham",
pages="66--78",
isbn="978-3-030-65883-0"
}

@inproceedings{hartisch2016quantified,
  title={Quantified integer programs with polyhedral uncertainty set},
  author={Hartisch, M. and Ederer, T. and Lorenz, U. and Wolf, J.},
  booktitle={International Conference on Computers and Games},
  pages={156--166},
  year={2016},
  organization={Springer}
}

@phdthesis{DissMichael,
  title={Quantified Integer Programming with Polyhedral and Decision-Dependent Uncertainty},
  author={Hartisch, M.},
  school={University of Siegen, Germany},
  volume={885},
  doi = {10.25819/ubsi/4841},
  year={2020},
}

@article{baggio2021multilevel,
  title={Multilevel approaches for the critical node problem},
  author={Baggio, A. and Carvalho, M. and Lodi, A. and Tramontani, A.},
  journal={Operations Research},
  volume={69},
  number={2},
  pages={486--508},
  year={2021},
  publisher={INFORMS}
}

@article{KM75,
title = {An analysis of alpha-beta pruning},
journal = {Artificial Intelligence},
volume = {6},
number = {4},
pages = {293-326},
year = {1975},
issn = {0004-3702},
doi = {https://doi.org/10.1016/0004-3702(75)90019-3},
url = {https://www.sciencedirect.com/science/article/pii/0004370275900193},
author = {D.E. Knuth and R.W. Moore}
}

@article{Plaat,
title = "Best-first fixed-depth minimax algorithms",
journal = "Artificial Intelligence",
volume = "87",
number = "1",
pages = "255 - 293",
year = "1996",
issn = "0004-3702",
doi = "https://doi.org/10.1016/0004-3702(95)00126-3",
url = "http://www.sciencedirect.com/science/article/pii/0004370295001263",
author = "A. Plaat and J. Schaeffer and W. Pijls and A. de Bruin"
}

@article{zhang2020unified,
  title={A unified framework for adjustable robust optimization with endogenous uncertainty},
  author={Zhang, Q. and Feng, W.},
  journal={AIChE Journal},
  volume={66},
  number={12},
  pages={e17047},
  year={2020},
  publisher={Wiley Online Library}
}

@article{chen2026robust,
  title={The robust selection problem with information discovery},
  author={Chen, Xiaoyu and Goerigk, Marc and Poss, Michael},
  journal={Discrete Applied Mathematics},
  volume={382},
  pages={277--292},
  year={2026},
  publisher={Elsevier}
}

@article{POSS2014836,
  title={Robust combinatorial optimization with variable cost uncertainty},
  author={Poss, M.},
  journal={European Journal of Operational Research},
  volume={237},
  number={3},
  pages={836--845},
  year={2014},
  publisher={Elsevier}
}

@InProceedings{WinandsHistory,
author="Winands, Mark H. M.
and van der Werf, Erik C. D.
and van den Herik, H. Jaap
and Uiterwijk, Jos W. H. M.",
editor="van den Herik, H. Jaap
and Bj{\"o}rnsson, Yngvi
and Netanyahu, Nathan S.",
title="The Relative History Heuristic",
booktitle="Computers and Games",
year="2006",
publisher="Springer Berlin Heidelberg",
address="Berlin, Heidelberg",
pages="262--272",
isbn="978-3-540-32489-8"
}

@article{Schaeffer,
 author = {Schaeffer, J.},
 title = {The History Heuristic and Alpha-Beta Search Enhancements in Practice},
 journal = {IEEE Trans. Pattern Anal. Mach. Intell.},
 issue_date = {1989},
 volume = {11},
 number = {11},
 year = {1989},
 issn = {0162-8828},
 pages = {1203--1212},
 numpages = {10},
 url = {http://dx.doi.org/10.1109/34.42858},
 doi = {10.1109/34.42858},
 acmid = {68878},
 publisher = {IEEE Computer Society},
 address = {Washington, DC, USA},
}

\begin{appendix}
\section{Multilevel Critical Node Problem\label{App::MCN}}
We consider the critical node problem as investigated in \cite{baggio2021multilevel}. Given a directed graph $G=(V,E)$ with a set of vertices $V$ and a set of edges $E$. Nodes can be affected by a viral attack, which triggers a cascade of infections that propagates via the graph neighborhood from node to node. Two agents act on graph $G$: The \textit{attacker} that can select a set of nodes she wants to infect and the \textit{defender} who tries to maximize the number of saved nodes. The defender has the ability to \textit{vaccinate} nodes before any infection occurs and to \textit{protect} a set of nodes after the attack took place. For each action (vaccination, infection, protection) a budget exists limiting the number of chosen nodes, where each action consumes one unit. Let $\Omega$, $\Phi$ and $\Lambda$ be the integer budgets for vaccination, infection and protection, respectively. For any node $v \in V$ binary variables $z_v$, $y_v$, and $x_v$ are used to indicated its vaccination, infection, and protection, respectively. Continuous variables $\alpha_v \in [0,1]$ are used to indicate whether node $v\in V$ is saved eventually. It is worth noting that optimal values of $\pmb{\alpha}$ will be binary. Only the variables $\pmb{y}$ are adversary variables. Using the notion of \cite{baggio2021multilevel} their domain is given by the budget constraint
\begin{equation}
\Xi =\{\pmb{y}\in \{0,1\}^V \mid \sum_{v \in V} y_v \leq \Phi\} \, .
\end{equation}
Hence, the adversary constraint system only contains the single budget constraint (making these instances simply restricted). The full problem can be stated as follows:

\begin{subequations}
\label{Model::MCN::QIP}
\begin{align}
\max_{\pmb{z} \in \{0,1\}^V} \min_{\pmb{y} \in \Xi} \max_{\substack{\pmb{x} \in \{0,1\}^V\\ \pmb{\alpha} \in [0,1]^V}}& \sum_{v \in V} \alpha_v\\
\textnormal{s.t. }  
&\sum_{v \in V} z_v \leq \Omega\\
&\sum_{v \in V} x_v \leq \Lambda\\
&\alpha_v \leq 1+z_v-y_v &&\forall v \in V\label{Model::MCN::QIP::Vaccination}\\
&\alpha_v \leq \alpha_u +x_v +z_v &&\forall(u,v)\in E\label{Model::MCN::QIP::Cascade}
\end{align}
\end{subequations}
This model exhibits a polyhedral uncertainty set.  Constraint \eqref{Model::MCN::QIP::Vaccination} ensures that infected nodes cannot be saved, unless they were vaccinated, and Constraint \eqref{Model::MCN::QIP::Cascade} describes the propagation of the infection to neighboring nodes that are neither vaccinated nor protected. This model corresponds exactly to the trilevel program presented in \cite{baggio2021multilevel} with the key difference that we are able to directly state this model in the QLP file format and use our solver, without having to dualize, reformulate or develop domain specific algorithms.

This problem also can be stated in a more intuitive way---directly modeling that vaccinated nodes shall not be attacked---via additional adversary constraints, yielding the decision-dependent uncertainty set 
\begin{equation}
\Xi(\pmb{z})=\{\pmb{y}\in \{0,1\}^V \mid \sum_{v \in V} y_v \leq \Phi,\ y_v + z_v \leq 1\  \forall v \in V \} \, .
\end{equation}
The arising adversary constraint system now depends on first-stage decision variable assignments and directly prohibits the infection of vaccinated nodes. In order to obtain a model with decision-dependent uncertainty set we replace the uncertainty set as shown above and substitute 
\begin{equation}
    \alpha \leq 1-y_v \quad \forall v \in V
\end{equation}
for Constraint \eqref{Model::MCN::QIP::Vaccination}. As this formulation only adds variable bound constraints on single adversary variables, this formulation also remains simply restricted, i.e., after the decision-maker variables $\pmb{z}$ have been assigned, an inadmissible adversary assignment can be detected by looking at the adversary constraints separately: Either the budget is exceeded or---in the case of decision-dependent uncertainty---a vaccinated node is infected.

\section{FA and LA Formulations of the BOBILib instances \label{app::bilevel::switch}}
Using the \texttt{inter-fire} instances as an example, we illustrate how bilevel instances can be formulated in two ways: with the follower acting as the adversary (FA), or with the leader acting as the adversary (LA). For this, we use that \texttt{inter-fire} instances are in fact the lower-bilevel problem arising in the original trilevel multilevel critical node problem (see Appendix~\ref{App::MCN}). Hence, the leader in this bilevel program is the attacker of the MCN problem, aiming at an attack that minimized the number of saved nodes and the problem can be verbalized as: 
\begin{quote}
``Can I (the attacker) find a way to attack nodes, such that for all protective actions by the defender, the number of saved nodes is minimized?''
\end{quote} 
The problem from the perspective of the attacker, i.e., where the attacker is the decision-maker, is given by

\begin{align*}
\min_{\pmb{y} \in \{0,1\}^V} \max_{(\pmb{x},\pmb{\alpha}) \in \mathcal{U}(\pmb{y})}& \sum_{v \in V} \alpha_v\\
\textnormal{s.t. } & \sum_{v \in V} y_v \leq \Phi ,
\end{align*}
where
\begin{equation*}
\mathcal{U}(\pmb{y})
=
\left\{
(\pmb{x},\pmb{\alpha})\in \{0,1\}^V \times \{0,1\}^V
\;\middle|\;
\begin{aligned}
    &\sum_{v\in V} x_v \leq \Lambda,
    \\
    &y_v+\alpha_v \leq 1
    \quad \forall v\in V,\\
    &\alpha_v-\alpha_u-x_v \leq 0
    \quad \forall (u,v)\in E.
\end{aligned}
\right\} 
\end{equation*}
describes the decision-dependent uncertainty set of the adversary (the follower/defender). This FA formulation is a natural description when viewing this bilevel problem standalone. However, having the MCN problem in mind, the same interaction can instead be viewed from the complementary perspective, in which the adversary is the attacker and the decision-maker is the defender. Then, the problem from the defender's perspective can be stated as: 
\begin{quote}
``For all the possible attacks, can I find a way to protect nodes such that the number of saved nodes is maximized?''
\end{quote}
Then $\pmb{y}$ variables are the adversary variables with uncertainty set $
\Xi=\{\pmb{y}\in \{0,1\}^V \mid \sum_{v \in V} y_v \leq \Phi\}$. Hence, the constraint system describing the domain of the adversary variables is only a knapsack constraint and no longer decision-dependent.

In order to implement this change in the respective instances, only three steps are necessary:
\begin{itemize}
\item[1.] interchange the adversary and the decision-maker constraint systems: In the QLP file move constraints from \texttt{SUBJECT TO} to  \texttt{UNCERTAINTY SUBJECT TO} and vice versa.
 \item[2.] interchange the responsibility of the variables, i.e., formerly adversary variables are now decision-maker variables, and vice versa. Hence, In the QLP file move variables from \texttt{EXISTS} to  \texttt{ALL} and vice versa.
\item[3.] change the optimization direction from minimization to maximization, since under QLP file format convention, the objective represents the objective of the decision-maker.
\end{itemize}
Importantly, the variable order remains unchanged. Thus, after exchanging the responsibility of the variables, the defender's decisions still occur after---and may depend on---the attacker's decision. As the parser of the solver \Equips requires a leading decision-maker variable, we also added a dummy binary variable to the beginning of the order. We note that for \Yasol, this change is not required. 

This transformation was applicable to all benchmark instances considered in Section~\ref{sec::exp::bilevel}, due to their similar structure. However, we want to point out that this transformation does not work in general and should be revised for the particular problem type at hand.

\section{The Multistage Robust Cutting-Stock Problem\label{app:ms-binpacking-models}}
\subsection{Problem Formulation\label{app:true-qip}}

In the multistage robust cutting-stock problem a set of $B$ bars is available, where bar $b\in [B]$ has length $L_b\in\mathbb{N}$. There are $I\in\mathbb{N}$ different sets of items and all items of type $i\in [I]$ have the same length $\lambda_i\in\mathbb{N}$. It is known beforehand that $A_i\in\mathbb{N}$ orders of type $i\in[I]$ will arrive. Thus, $T = \sum_{i\in [I]} A_i$ orders must be processed over $T$ stages. At each stage exactly one item is ordered and must be assigned immediately to a bar. Decisions may depend on the history observed so far, but not on future item arrivals. The objective combines three cost components: usage costs $U\in\mathbb{Q}_{\geq 0}$ for each used bar, setup costs $S\in\mathbb{Q}_{\geq 0}$ for switching to another bar, and waste costs $W\in\mathbb{Q}_{\geq 0}$ for unusable leftover, i.e., bars with remaining length not exceeding the threshold length $R\in\mathbb{Q}_{\geq 0}$. To model this problem, we use decision variables $a_{b,i}^t \in \{0,1\}$, indicating whether an item of type $i$, arriving at stage $t$, is assigned to bar $b$. In the first decision stage, a predetermined item $i_0\in [I]$ must be assigned to one of the bars. Adversary variables $\xi_i^t\in\{0,1\}$ indicate for the remaining $\tilde{T}=T-1$ stages whether an item of type $i\in[I]$ is ordered in stage $t \in [\tilde{T}]$. Hence, the domain of these variables is given by
$$\mathcal{U}=\{\pmb{\xi} \in \{0,1\}^{\tilde{T}\times I} : \vert\vert \pmb{\xi}^t \vert \vert_1 =1\ \forall t \in [\tilde{T}] \wedge \sum_{t\in[\tilde{T}]} \xi^t_i = A_i -\mathds{1}_{i=i_0}\ \forall i \in [I]\}$$
and the full adversary constraint system consists of $\tilde{T}+I$ constraints. We write
$$\mathcal{U}_t=\mathcal{U}_t(\xi^1,\ldots,\xi^{t-1})=\{\pmb{\xi} \in \{0,1\}^{I} : \vert\vert \pmb{\xi} \vert \vert_1 =1\  \wedge \sum_{t'\in[t]} \xi^{t'}_i \leq A_i- \mathds{1}_{i=i_0}\ \forall i \in [I]\}$$ 
for the universal domain relevant in stage $t\in[\tilde{T}]$. This shows that these instances can be viewed as simply restricted under the assumption that unit propagation is carried out after every made assignment. Then, the variable order (using quantification to indicate adversary ($\forall$) and decision variables ($\exists$)) is given by 
\[
    \exists a_0\in\{0,1\}^{B\times I}
    \
    \forall \xi^1\in\mathcal U_1
    \
    \exists a^1\in\{0,1\}^{B\times I}
    \ \cdots \
    \forall \xi^{\tilde{T}}\in\mathcal U_{\tilde{T}}
    \
    \exists a^{\tilde{T}}\in\{0,1\}^{B\times I}.
\]
The decision-maker constraint system is given by 
\begin{align}
    \sum_{b\in [B]} a_{b,i}^0 = \mathds{1}_{i=i_0}
    & \qquad \label{model::cs::0}\\
    \sum_{b\in [B]} a_{b,i}^t = \xi_i^t
    & \qquad
    \forall t=[\tilde{T}],\ i\in [I]\label{model::cs::1}\\
    \sum_{t\in[T]}\sum_{i\in [I]}
        \lambda_i a_{b,i}^t
    \le L_b
    & \qquad
    \forall b\in [B]\label{model::cs::2} ,
\end{align}
where Constraint~\eqref{model::cs::0} ensures that the initial item is assigned to some bar,  Constraint~\eqref{model::cs::1} ensures that the ordered item is processed in stage $t$ and Constraint~\eqref{model::cs::2} is the bar-capacity constraint.

The objective function can be states as
\[
    \min\;
    \sum_{b\in [B]} W w_b(a)
    +
    \sum_{t\in[\tilde{T}]}\sum_{b\in [B]} S s_b^t(a)
    +
    \sum_{b\in [B]} U u_b(a) .
\]
Here, 
\[
w_b(a)=
\begin{cases}
0,
    & \text{if }
 
      \sum_{t\in[T]}\sum_{i\in [I]}\lambda_i a_{b,i}^t
      \le L_b-R,\\
L_b-
      \sum_{t\in[T]}\sum_{i\in [I]}\lambda_i a_{b,i}^t,
    & \text{otherwise,}
\end{cases}
\]
calculates the amount of waste left of bar $b$,
\[
s_b^t(a)=
\begin{cases}
1,
    & \text{if bar $b$ is used in stage $t\in[\tilde{T}]$ but was not used in stage $t-1$,}\\
0,
    & \text{otherwise,}
\end{cases}
\]
and
\[
u_b(a)=
\begin{cases}
0, & \text{if bar $b$ is never used,}\\
1, & \text{otherwise.}
\end{cases}
\]
These functions are linearized by adding auxiliary variables $w_b\in [0,R]$, $s_b^t\in\{0,1\}$, and $u_b \in \{0,1\}$, respectively, which are added as decision variables in the final stage. To correctly model the waste, an additional auxiliary variable $r_b\in \{0,1\}$ is used to indicate whether bar $b$ is reusable (i.e., it is no waste). Then, the full decision-maker constraint system is given by

\begin{align}
\textnormal{\eqref{model::cs::0}--\eqref{model::cs::2}} \notag\\
     \sum_{t\in [T]}\sum_{i\in [I]} a_{b,i}^t \leq T u_b
    & \qquad
    \forall b \in [B]\label{model::csf::3}\\
    \sum_{t\in[T]}\sum_{i\in [I]}
        \lambda_i a_{b,i}^t
    + w_b \le L_b
    & \qquad
    \forall b\in [B]\label{model::csf::waste1} \\ 
    \sum_{t\in[T]}\sum_{i\in [I] }
        \lambda_i a_{b,i}^t
    + w_b + L_b r_b \ge L_b
    & \qquad
    \forall b\in [B]\label{model::csf::waste2} \\ 
    \sum_{t\in[T]}\sum_{i\in [I]}
        \lambda_i a_{b,i}^t
    + R r_b \le L_b
    & \qquad
    \forall b\in [B]\label{model::csf::waste3} \\ 
 w_b + R r_b \leq R & \qquad
    \forall b\in [B] \label{model::csf::waste4} \\
\sum_{i\in [I]} a_{b,i}^t -\sum_{i\in [I]} a_{b,i}^{t-1} \leq s^t_b
    & \qquad
    \forall t \in [\tilde{T}], b\in [B],\label{model::csf::setup}
\end{align}
where Constraint~\eqref{model::csf::3} detects the usage of each bar, Constraints~\eqref{model::csf::waste1}--\eqref{model::csf::waste4} linearize the calculation of the amount of waste for each bar, and Constraint~\eqref{model::csf::setup} ensures that setup costs must be paid if the used bar changes from stage $t-1$ to stage $t$.
\subsection{Instance Generation\label{app:instance-generation}}

For every generated instance, the waste threshold is set to $10$ and the cost parameters are set to  $W=1$, $S=5$, and $U=4$. We created each $10$ instances for every $T\in \{5,6,\ldots,12\}$ and $I\in \{2,3,4,5\}$ where we set $B=\lfloor \frac{T}{2} \rfloor$, resulting in $320$ instances. For each bar $b\in B$, its length $L_b$ is sampled independent and uniformly from $\{30,\ldots,100\}$. Item lengths $\lambda_i$, $i\in I$, are then sampled independent and uniformly from
$
\{R,\ldots,\lfloor \frac{1}{T}\sum_{b\in[B]}L_b\rfloor\}$. The number of orders of each item type is generated by constructing all nonincreasing positive integer vectors
$(A_1,\ldots,A_{|I|}) \in \mathbb{Z}_{>0}^{I}$ with $\sum_{i\in I} A_i = T$, of which one is selected uniformly at random. Finally, the type of the initial item is also selected at random.

We note that this generation procedure does not ensure feasibility of all instances, as it might be impossible to assign all items to bars, respecting the bar lengths. In particular, for $11$ of the $320$ created instances no solution exists.

\subsection{Single-stage and Two-Stage Approximations}
\label{app:single-two-stage}

The single-stage approximation ignores the order in which items may arrive and computes a fixed assignment plan $\bar a_{b,i} \in \mathbb{Z}_{\ge 0}$, where $\bar a_{b,i}$ denotes how many items of type $i\in[I]$ are assigned to bar $b\in[B]$ via the following IP:
\begin{align*}
\min \	& \sum_{b\in [B]} W\cdot w_{b}(a) + S \left(\sum_{b\in [B]}u_{b}(a)-1\right) + \sum_{b\in [B]} U\cdot u_{b}(a) \\	\textnormal{s.t.}\ &\sum_{i\in [I]}\lambda_{i}\cdot a_{b,i} \leq L_{b} \qquad\forall b \in B\\
	& \sum_{b\in [B]} a_{b,i} = A_i \qquad \forall i \in I\\
	& a_{b,i} \in \mathbb{Z}_{\ge 0} \qquad\forall b \in [B] \text{, }i \in [I]\\
\end{align*}
To compute setup costs, it assumes that bars are used until all assigned items are processed.

The two-stage robust approximation also produces a fixed aggregate plan, but
requires that the cost of realizing this plan are minimal for the worst-case order sequence. In particular, while bar usage and waste cost can be evaluated on the fixed plan, the setup cost are evaluated against the worst-case order sequence.  The problem can be stated as
\begin{align*}
\min_{\bar a \in \mathbb{Z}_{\ge 0}^{B\times I}}\ &  \max_{\xi\in\mathcal U}	\ \min_{r\in\{0,1\}^{T\times B\times I}}\sum_{b \in[B]} W\cdot w_{b}(a) + \sum_{t \in[T]} \sum_{b \in[B]} S\cdot s^t_b(r) + \sum_{b \in[B]} U\cdot u_{b}(a) \\
\textnormal{s.t.} \
&\sum_{i\in [I]}\lambda_{i}\cdot a_{b,i} \leq L_{b} \qquad\forall b \in [B]\\
	& \sum_{b\in [B]} a_{b,i} = A_i \qquad \forall i \in [I]\\
& \sum_{t\in[T]} r^t_{b,i} = a_{b,i} \qquad \forall b \in [B] \text{, }i \in [I]\\
& \sum_{b\in [B]} r^t_{b,i} = \xi^t_i \qquad \forall t \in [T] \text{, }i \in [I], 
\end{align*}
where $r$ represents the scenario-dependent realization of the fixed plan. While this problem takes different order sequences into account, it neglects that for different order sequences different assignments of items to bars can be realized.

Both approaches are evaluated by enforcing the obtained plan
$\bar a$ in the true multistage model. The validation model is the true multistage model as shown in Section~\ref{app:true-qip} 
with the additional decision-maker constraints 
\[
    \sum_{t\in [T]} a_{b,i}^t
    =
    \bar a_{b,i}
    \qquad
    \forall b\in [B],\ i\in [I] .
\]
Solving this validation problem (with \Yasol) yields the actual worst-case multistage performance of the fixed plan.

\end{appendix}
\end{document}